\documentclass[english, 12pt]{article}
\usepackage[T1]{fontenc}
\usepackage{geometry}
\usepackage[latin9]{inputenc}
\usepackage{amsthm}
\usepackage{amsmath}
\usepackage{amssymb}
\usepackage{authblk}
\usepackage[margin=2cm, labelfont = bf]{caption}
\usepackage{setspace}
\usepackage{esint}
\usepackage{enumerate}
\usepackage{etoolbox}
\usepackage{url}
\usepackage{graphicx, epstopdf}
\usepackage[english]{babel}
\usepackage{xcolor}
\usepackage{multicol}
\usepackage{soul}
\usepackage{tikz}
\usetikzlibrary{automata,arrows,positioning,calc}

\numberwithin{equation}{section}

\makeatletter
\theoremstyle{plain}
\newtheorem{thm}{\protect\theoremname}[section]
\ifx\proof\undefined
\newenvironment{proof}[1][\protect\proofname]{\par
\normalfont\topsep6\p@\@plus6\p@\relax
\trivlist
\itemindent\parindent
\item[\hskip\labelsep
\scshape
#1]\ignorespaces
}{%
\endtrivlist\@endpefalse
}
\providecommand{\proofname}{Proof}
\fi
\theoremstyle{plain}
\newtheorem{lem}[thm]{\protect\lemmaname}
\theoremstyle{plain}
\newtheorem{prop}[thm]{\protect\propositionname}
\theoremstyle{plain}

\theoremstyle{definition}
\newtheorem{defn}[thm]{\protect\definitionname}
\theoremstyle{remark}
\newtheorem{rem}[thm]{\protect\remarkname}
\theoremstyle{plain}

\theoremstyle{definition}

\numberwithin{figure}{section}

\makeatother

\usepackage{babel}
\providecommand{\conjecturename}{Conjecture}
\providecommand{\corollaryname}{Corollary}
\providecommand{\definitionname}{Definition}
\providecommand{\examplename}{Example}
\providecommand{\lemmaname}{Lemma}
\providecommand{\propositionname}{Proposition}
\providecommand{\remarkname}{Remark}
\providecommand{\theoremname}{Theorem}

\usepackage{mathrsfs}

\begin{document}


\title{\vspace{-15pt}The supercooled Stefan problem with \\transport noise: weak solutions and blow-up}

\author{Sean Ledger}
\author{Andreas S{\o}jmark\footnote{AS would like to thank Sergey Nadtochiy as well as Sergey, Misha Shkolnikov, and Christa Cuchiero for organising, respectively, the session on Probabilistic methods for Stefan-type equations at the SPA conference in Lisbon and the workshop on Laplacian Growth Models at IMSI in Chicago. Early versions of the results in this paper were first presented on these occasions and they took their current form following fruitful discussions with the organisers and other participants.}}
\affil{Department of Statistics, London School of Economics}

\date{March 9, 2026}

\maketitle
\vspace{-15pt}

\begin{abstract} We derive two weak formulations for the supercooled Stefan problem with transport noise on a half-line: one captures a continuously evolving system, while the other resolves blow-ups by allowing for jump discontinuities in the evolution of the temperature profile and the freezing front. For the first formulation, we establish a probabilistic representation in terms of a conditional McKean--Vlasov problem, and we then show that there is finite time blow-up with positive probability when part of the initial temperature profile is supercooled below a critical value. On the other hand, the system is shown to evolve continuously when the initial profile is everywhere above this value. In the presence of blow-ups, we show that the conditional McKean--Vlasov problem provides global solutions of the second weak formulation. Finally, we identify a solution of minimal temperature increase over time and we show that its discontinuities are characterised by a natural resolution of emerging instabilities with respect to an infinitesimal external heat transfer.
 \end{abstract}

\vspace{2pt}

\section{Introduction}

This work is concerned with a stochastic version of the supercooled Stefan problem on the positive half-line. In the deterministic setting, the classical formulation of this problem is given by the following free boundary problem:
\begin{equation}\label{eq:classical_soln}
	\left\{
	\begin{aligned}
		\partial_{t}v(t,x) &= \displaystyle \kappa \partial_{xx} v(t,x) \;\; \text{for} \;\; x\in (s(t),\infty)  \\[3pt]
		v(t,x)&=v_f \;\; \text{for} \;\; x\in [0,s(t)) \\[3pt]
		v(t,s(t)) &=v_f, \quad   \partial_x v(t,s(t)) = -\lambda \dot{s}(t),
	\end{aligned}
	\right.
\end{equation}
with initial conditions $s(0)=s_0$ and $v(0,\cdot)=v_0$, for some $s_0\in[0,\infty)$ and $v_0\in L^1(\mathbb{R})$, where $v_0(x) \leq v_f$ for a.e.~$x\in[0,\infty)$, and $\kappa, \lambda >0$ and  $v_f\in \mathbb{R}$ are given constants.

To explain the formulation, consider a liquid that occupies the positive half-line and let $v(t,x)$ denote its temperature at the spatial position $x\in[0,\infty)$, at time $t\geq 0$. Initially, the liquid is frozen on $[0,s_0]$, while it is supercooled on $(s_0,\infty)$, meaning that $v_0(x)<v_f$ for all $x\in (s_0,\infty)$, where $v_f$ denotes the equilibrium freezing temperature of the liquid. Assuming that the phase change is isothermal and that there is zero heat flux in the frozen state, the temperature equals $v_f$ at the interface $x=s(t)$, where $s(t)$ denotes the current position of the freezing front, and we have a constant temperature $v(t,x)=v_f$ on $[0,s(t)]$. Now let the properties of the liquid be described by its thermal conductivity $k$, latent heat $\ell$, mass density $\varrho$, and specific heat capacity $c$. Then, the laws of thermodynamics give that the evolution of the temperature profile $v(t,x)$ and the freezing front $s(t)$ are governed by \eqref{eq:classical_soln} with
\begin{equation}\label{eq:clambda_kappa}
	\lambda = \frac{\varrho \ell}{k}>0 \quad \text{and} \quad \kappa= \frac{k}{c\varrho}>0.
\end{equation}

While most of the results in this paper can also be of interest for \eqref{eq:classical_soln}, we are chiefly interested in the stochastic version that arises when introducing a Brownian transport noise into the evolution of the temperature profile. Formally, this new problem reads as
\begin{equation}\label{eq:classical_noise}
	\left\{
	\begin{aligned}
		\mathrm{d} v(t,x) &= \displaystyle \kappa \partial_{xx} v(t,x) \mathrm{d}t+\theta \partial_x v(t,x) 	\mathrm{d}  W_{t}   \;\;\, \text{for} \;\; x\in ( s(t),\infty)  \\[3pt]
		v(t,x)&= v_f \;\; \text{for} \;\; x\in[0, s(t)) \\[3pt]
		v(t,s(t)) &= v_f, \quad   \partial_x v(t,s(t)) = - \lambda \dot{s}(t) ,
	\end{aligned}\right.
\end{equation}
for a given noise parameter $\theta \neq 0$, with $v(0,x)=v_0(x)$ for $x\in (s_0,\infty)$ and $s(0)=s_0$. Note that $\theta =0$ takes us back to \eqref{eq:classical_soln}. We shall always assume $|\theta|<\sqrt{2\kappa}$ so that \eqref{eq:classical_noise} defines a free boundary problem for which the temperature profile $v$ follows a \emph{parabolic} stochastic PDE (see e.g.~\cite[Section 1.2]{Brzezniak}) in line with the parabolic nature of the deterministic problem. Furthermore, for simplicity of notation, we will take the equilibrium freezing temperature $v_f$ to be $v_f=0$ throughout.
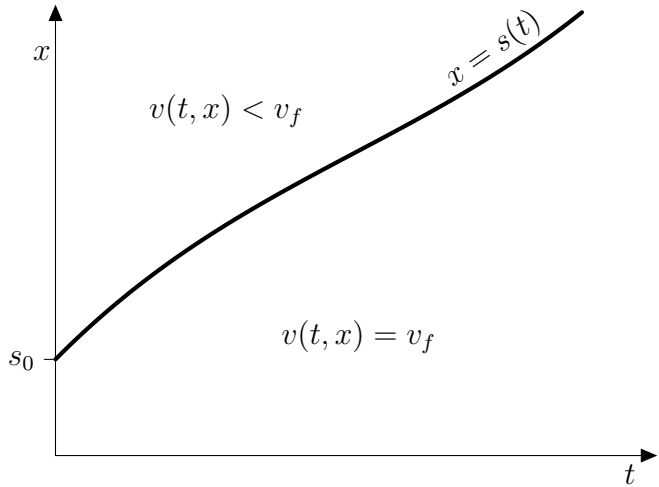
\begin{figure}
	\centering
	\begin{tikzpicture}[line cap=round,line join=round,>=triangle 45,x=1cm,y=1cm]
		\clip(-1.21,-0.65) rectangle (11.46,6.72);
		\draw [->] (0,0) -- (0,6);
		\draw [->] (0,0) -- (8,0);
		\draw (0,1.28)-- (-0.15,1.28);
		\draw[line width=1.6pt, smooth,samples=100,domain=0.0:7.0] plot(\x,{0.01*(\x)^3-0.12*(\x)^2+1.01*(\x)+1.28});
		\draw (-0.44,5.59) node[anchor=north west] {$ x $};
		\draw (7.42,0.02) node[anchor=north west] {$ t $};
		\draw (-0.78,1.53) node[anchor=north west] {$ s_0 $};
		\draw (1.1,4.95) node[anchor=north west] {$ v(t,x) < v_f $};
		\draw (2.86,1.94) node[anchor=north west] {$ v(t,x) = v_f $};
		\draw (4.9,5.2) node[rotate=35, anchor=north west] {$ x = s(t) $};
	\end{tikzpicture}\vspace{-10pt}
	\caption{Domain of the free boundary problems \eqref{eq:classical_soln} and \eqref{eq:classical_noise}. One should have in mind a continuously differentiable boundary for \eqref{eq:classical_soln} versus something akin to Minkowski's question-mark function for \eqref{eq:classical_noise}.}\label{fig:cont_domain}
\end{figure}

A nonzero parameter $\theta \neq 0$ in \eqref{eq:classical_noise} can serve as a model for uncertainty in the measurement of the true temperature. Similarly, it provides a model for the presence of a random environment affecting the current temperature by Brownian fluctuations. Later, we shall study a probabilistic representation which highlights the following interpretation: at the microscopic level, the heat diffusion is described by a mean-field of Brownian `heat particles' which are correlated through a common Brownian motion $W$ that then appears as a stochastic transport term in the evolution of the temperature profile at the macroscopic level.

We also mention that $\theta\neq 0$ in \eqref{eq:classical_noise} arises naturally in the analysis of recent models for financial contagion (see \cite{hambly_ledger_sojmark_2018, LS, nadtochiy_shkolnikov_2017, nadtochiy_shkolnikov_2020}) when introducing common exposures between the firms. Similarly, a version of \eqref{eq:classical_noise} with re-insertion of mass would appear in the mean-field limit of networked integrate-and-fire models (see \cite{DIRT_AAP, DIRT_SPA}) if the noisy parts of the neurons' membrane potentials are taken to be correlated through a common noise.

When discussing \eqref{eq:classical_noise}, it should be stressed that there has already been significant interest in various stochastic versions of the Stefan problem, see in particular \cite{Barbu, Hambly_Kalsi, Keller-Ressel, Kim1,  Kim2, muller}. Out of this literature, \cite{muller} is the only work that considers a stochastic perturbation of transport type, as in the case of \eqref{eq:classical_noise}, but the conditions at the interface differ somewhat from the actual Stefan problem and, as noted in the paper, this appears to be crucial for the particular framework \cite[p.~2339]{muller}. More importantly, none of these existing works are interested in the \emph{supercooled} specification, so our analysis is novel in that regard.

As concerns the original (i.e., noiseless) supercooled Stefan problem, a breakthrough was recently achieved in \cite{DNS}. Given a bounded right-continuous initial profile $v_0$ with at most finitely many changes of monotonicity on compacts, it was shown that there exists a \emph{unique} global `physical' solution (see Remark \ref{rem:physical_minimal}) to a probabilistic McKean--Vlasov reformulation of the problem that allows for finite time blow-ups and provides, at any time $t\geq 0$, a classical solution of \eqref{eq:classical_soln} on a subsequent time interval $(t,t+\varepsilon)$. In later work, \cite{cuchiero} focused on the notion of a \emph{minimal} solution to the aforementioned McKean--Vlasov formulation and showed that this (uniquely defined) minimal solution is physical in the same sense as in \cite{DNS}. Furthermore, \cite{bayraktar} has derived conditions under which the McKean--Vlasov problem yields a classical solution of \eqref{eq:classical_soln} for all time. This extends ideas from earlier works \cite{Fasano1, Fasano2} in the PDE literature that studied local and global classical solutions of \eqref{eq:classical_soln} when formulated for a supercooled liquid on a bounded interval.

The analysis of this paper will differ markedly from the above, as we focus on the stochastic problem \eqref{eq:classical_noise}. Moreover, we take a different route by first deriving suitable weak formulations, starting from a continuously evolving system without enforcing any differentiability or absolute continuity constraints. Then, we show that this formulation admits a probabilistic representation, which now takes the form of a \emph{conditional} McKean--Vlasov problem due to the stochastic perturbation. Moreover, as one of the main contributions of this paper, we show that any such solution blows up in finite time with positive probability if the initial profile is supercooled below $\lambda \kappa$ at any given point (assuming also the profile is right-continuous at that point). Here a blow-up refers to a breakdown of the weak formulation. This is related, at least in spirit, to the results of \cite{dozzi} and \cite{mueller_sowers, mueller} on finite time blow-up with positive probability for different stochastic perturbations of the heat equation.

Finally, we proceed to derive a more general weak formulation, which allows for jump discontinuities to emerge, and for which we can establish global existence by means of the conditional McKean--Vlasov problem. As our second main contribution, we single out one of these probabilistic solutions as the solution of minimal total temperature increase over time, in the pathwise sense, and we confirm that its jump discontinuities obey a natural selection principle. Our approach to this part of the paper not only yields novel results for \eqref{eq:classical_noise} but also provides new perspectives on the recent work of \cite{cuchiero, DNS} for the deterministic problem \eqref{eq:classical_soln}.

\section{Weak formulations}\label{sect:results_supercooled}

By integrating $v(t,\cdot)$ from \eqref{eq:classical_soln} against a suitable test function, enforcing
$v(t,s(t))=0$ and $\partial_x v(t,s(t)) = -\lambda\dot{s}(t)$, and then integrating by parts, we arrive at a natural notion of weak solution that implicitly encodes these boundary conditions. After first introducing some notation, we state this weak formulation in Definition \ref{def:weak_cont} below.

Throughout, we work on a suitable complete probability space $(\Omega , \mathcal{F},\mathbb{P})$. We shall write $\textbf{C}_\mathbb{R}$ for the space of real-valued continuous paths on $[0,\infty)$, and we write $\textbf{C}_{L^1(\mathbb{R})}$ for the space of $L^1(\mathbb{R})$-valued continuous paths on $[0,\infty)$ in the sense that
$u \in \textbf{C}_{L^1(\mathbb{R})}$ if and only if $u(t,\cdot) \in L^1(\mathbb{R})$ for all $t\geq0$ and $t\mapsto \langle u(t,\cdot ) , \phi \rangle$ is in $\textbf{C}_\mathbb{R}$ for all $\phi \in C_b(\mathbb{R})$. Moreover, we let $\textbf{C}^\uparrow_\mathbb{H}$ denote the space of increasing continuous paths with values in the positive half-line $\mathbb{H}:=[0,\infty)$. Given the initial data $s_0\in \mathbb{H}$ and $v_0 \in L^1 \cap L^\infty(\mathbb{R})$ with $\mathrm{supp}(v_0)\subseteq [s_0,\infty)$, we then define an admissible class for the paths of any candidate solution by
\begin{align*}
	\textbf{C}(v_0,s_0):= \Bigl\{  &(v,s) \in \textbf{C}_{L^1(\mathbb{R})}\times \textbf{C}^\uparrow_{\mathbb{H}} :  s(0)=s_0, \,v(0,\cdot)=v_0 \\ &\;\;\;\mathrm{supp}(v(t,\cdot))\subseteq [s(t),\infty), \; \Vert v(t,\cdot)\Vert_{L^\infty} \leq \Vert v_0 \Vert_{L^\infty} \;\text{for all } t \geq 0     \Bigr\}.
\end{align*}
Next, we then define $\mathfrak{C}(v_0,s_0)$ as the space of $\mathcal{F}_t$-adapted stochastic processes with paths in $\textbf{C}(v_0,s_0)$, where $\mathcal{F}_t$ denotes the complete, right-continuous filtration generated by the Brownian motion $W$.

\begin{defn}[Continuous weak solutions]\label{def:weak_cont} Let $s_0\in\mathbb{H}$ and $v_0\in L^1\cap L^\infty(\mathbb{R})$ with $\mathrm{supp}(v_0)\subseteq[s_0,\infty)$ and $v_0\leq 0$ (a.e.). We say that $(v,s) \in \mathfrak{C}(v_0,s_0)$ is a (global) continuous weak solution to the supercooled Stefan problem \eqref{eq:classical_noise} if, almost surely,
	\begin{align}\label{eq:weak_cont}
		\int_{s(t)}^{\infty}&v(t,x) \phi(t,x) \mathrm{d}x -  \int_{s_0}^{\infty}\! v_0(x) \phi(0,x) \mathrm{d}x = \int_0^t \int_{s(r)}^{\infty }v(r,x)\partial_r \phi(r,x) \mathrm{d}x \mathrm{d}r\nonumber \\
		&+\kappa\!\int_0^t\! \int_{s(r)}^{\infty }\! v(r,x) \partial_{xx}\phi(r,x) \mathrm{d}x\mathrm{d}r +\theta\!\int_0^t\! \int_{s(r)}^{\infty } \!v(r,x) \partial_x\phi(r,x) \mathrm{d}x\mathrm{d} W_r  \nonumber \\ 
		&\qquad\qquad\quad + \lambda\kappa\! \int_0^t \phi(r,s(r))\mathrm{d} s(r),
	\end{align}
	for all $t\geq 0$ and all test functions $\phi 
	\in C^{1,2}_b([0,\infty)\times \mathbb{R})$. For an $\mathcal{F}_t$-stopping time $\tau$, we say that $(v,s)$ is a continuous weak solution on $[0,\tau)$ if the above holds with $t\land \tau$ in place of $t$.
\end{defn}

We note that the Dirichlet boundary condition of \eqref{eq:classical_noise} is implicitly imposed by the fact that the boundary points are allowed to be in the support of the test functions. Moreover, the shape of the heat flux across the interface is dictated by the line integral along the freezing front in \eqref{eq:weak_cont}. This specification only requires the front to be continuous and increasing, unlike the derivative constraint in \eqref{eq:classical_noise} which can never be satisfied when $\theta \neq 0$. Taking $\phi \equiv 1$ in  \eqref{eq:weak_cont}, we see that the weak formulation implies
\begin{equation}\label{eq:energy_balance}
	s(t) -  s_0 = \displaystyle \frac{1}{\lambda \kappa} \Bigr( \displaystyle\int_{s(t)}^\infty \!v(t,x)\mathrm{d}x -  \displaystyle\int_{s_0 }^\infty \!v_0(x)\mathrm{d}x \Bigr),
\end{equation}
for all $t\geq 0$. This constraint enforces conservation of heat, as can also be derived from first principles in the following way. During the phase transition at the interface, the advance of the freezing front by an amount $\Delta s$ releases $\varrho \ell \Delta s$ units of latent heat. At the same time, an increase of the total temperature by $\Delta u$ absorbs $c \varrho \Delta u $ units of sensible heat. By the first law of thermodynamics, we must therefore have
\begin{equation}\label{eq:energy_bal_first_princip}
	\Delta s = \frac{c \varrho \Delta u}{\varrho \ell }= \frac{c}{\ell}  \Delta u=  \frac{1}{\lambda \kappa }\Delta u,
\end{equation}
where we have used the definition of $\lambda$ and $\kappa$ in \eqref{eq:clambda_kappa}. Since the total increase in temperature on $[0,t]$ is
$
\int_{s(t)}^\infty \!v(t,x)\mathrm{d}x - \int_{s_0 }^\infty \!v_0(x)\mathrm{d}x,
$
the conclusion follows.

In view of \eqref{eq:energy_balance}, we are led to consider a probabilistic representation of \eqref{eq:weak_cont} where the freezing front is linked to the (stochastic) loss of mass for a suitable Brownian particle that is absorbed upon hitting the freezing front. Our first result makes this precise and confirms that any weak solution in the sense of Definition \ref{def:weak_cont} can be characterised in this way.

\begin{thm}[Probabilistic McKean--Vlasov representation]\label{thm:prob_rep_cont}
	If $(v, s)$ is a (global) continuous weak solution in the sense of Definition \ref{def:weak_cont}, then $(v,s)$ is characterised by
	\begin{equation}\label{eq:prob_rep_0}
		\left\{
		\begin{aligned}
			v(t,x)\mathrm{d}x &=\int_{s_0}^\infty \mathbb{P}(X^y_t \in \mathrm{d}x , \,t<\tau^y \mid \mathcal{F}_t) v_0(y)\mathrm{d}y \\
			s(t) &= s_0 - \frac{1}{\lambda \kappa} \int_{s_0}^\infty\mathbb{P}(\tau^y\leq t \mid \mathcal{F}_t)v_0(y)\mathrm{d}y,
		\end{aligned}
		\right.
	\end{equation}
	for all $t\geq0$, almost surely, where
	\begin{equation}\label{eq:prob_rep}
		\left\{
		\begin{aligned}
			\tau^y &= \inf\{ t \geq 0 : X^y_t \leq s(t) \} \\[3pt]
			X^y_t &= y + \textstyle\sqrt{2\kappa-\theta^2} B_t + \theta W_t,
		\end{aligned} \right.
	\end{equation}
	for a Brownian motion $B$ independent of $W$. If $(v,s)$ is a local continuous weak solution on $[0,\tau)$, for some $\mathcal{F}_t$-stopping time $\tau$, then the above holds on $[0,\tau)$.
\end{thm}

Observe that \eqref{eq:prob_rep_0}--\eqref{eq:prob_rep} amounts to a conditional McKean--Vlasov problem, since the freezing front $t\mapsto s(t)$ is determined by the conditional law of the hitting time $\tau^y$ for each $y$ in the support of $v_0$. We shall study this McKean--Vlasov problem in its own right in Section \ref{sect:McKean}, and we stress that all the results of the present section rely on this analysis. For $\theta=0$, the McKean--Vlasov point of view is explored in \cite{bayraktar, cuchiero, DNS, hambly_ledger_sojmark_2018, nadtochiy_shkolnikov_2017}.

Already in the case $\theta=0$, Theorem \ref{thm:prob_rep_cont} and the same arguments as in \cite[Theorem 1.1]{hambly_ledger_sojmark_2018} tell us that continuous weak solutions \eqref{eq:weak_cont} can fail to exist globally in time. Thus, we shall develop a more general weak formulation in the next subsection, without the assumption of continuity, and we will then show that global weak solutions may be obtained from the associated conditional McKean--Vlasov problem. This is similar to \cite{DNS} who studied the classical formulation \eqref{eq:classical_soln} in the presence of jump discontinuities, by taking an unconditional analogue of the aforementioned McKean--Vlasov problem (with $\theta=0$) as the definition of solution. Here we are chiefly interested in $\theta \neq 0$ and we take a different route by deriving appropriate global weak formulations.

\subsection{Allowing for temperature discontinuities}\label{sect:temp_discont}

From here onwards, we shall only insist on a c\`adl\`ag structure for the temperature profile and the freezing front. At a discontinuity time $t>0$, there will be an instantaneous phase transition advancing the freezing front by a non-zero amount $\Delta s(t):= s(t) - s(t-)$ with the temperature jumping from supercooled, i.e.~$v(t-,x) <0 $, to frozen, i.e.~$v(t,x)=0$, for all $x\in (s(t-),s(t)]$. Away from the interface, we assume that the temperature profile is unaffected, meaning that $v(t,x)=v(t-,x)$ for all $x>s(t)$. For any given $\phi \in C^{1,2}_b([0,\infty)\times \mathbb{R})$, we then have that the temperature profile tested against $\phi$ undergoes a jump of size
\begin{equation}\label{eq:tested_jump_size}
	\int_{s(t)}^{\infty}v(t,x) \phi(t,x) \mathrm{d}x - 	\int_{s(t-)}^{\infty}v(t-,x) \phi(t,x) \mathrm{d}x  = - 	\int_{s(t-)}^{s(t)}v(t-,x) \phi(t,x) \mathrm{d}x.
\end{equation}
At the same time, the flux across the interface is dictated by the corresponding line integral along the freezing front, namely $\int_0^t\phi(r,s(r-))\mathrm{d}s(r)$, which undergoes a jump of size $\phi(t,s(t-))\Delta s(t)$. Since conservation of heat energy must still hold, the constraint \eqref{eq:energy_bal_first_princip} has to apply also at jump times, yielding
\begin{equation}\label{eq:simplified_energy_balance}
	\Delta s(t) = \frac{c}{\ell }\Delta u(t) = - \frac{1}{\lambda \kappa} \int_{s(t-)}^{s(t)} \!v(t-,x)\mathrm{d}x,
\end{equation}
due to $v(t,x)=v(t-,x)$ for $x>s(t)$ (as in \eqref{eq:tested_jump_size} with $\phi \equiv 1$). This of course agrees with enforcing \eqref{eq:energy_balance} at $t$ and $r< t$, 
taking the difference, and sending $r\uparrow t$. It follows that the aforementioned jump of the line integral is given by a multiple $1/\lambda \kappa$ of
\begin{equation}\label{eq:line_int_jump}
	-  \int_{s(t-)}^{s(t)} \!v(t-,x)\phi(t,s(t-))\mathrm{d}x.
\end{equation}
Since \eqref{eq:tested_jump_size} and \eqref{eq:line_int_jump} are in general \emph{not} equal, the continuous weak formulation \eqref{eq:weak_cont} fails to make sense in the presence of temperature discontinuities. By correcting for the difference between \eqref{eq:tested_jump_size} and \eqref{eq:line_int_jump}, we arrive at a consistent weak formulation in Definition \ref{def:weak_jumps} below. For a stylised illustration of the emergence of a temperature discontinuity, see Figure \ref{fig:jump_illustration}.

\begin{figure}
	\centering
	\hspace{-25pt}\includegraphics[width=0.43\textwidth]{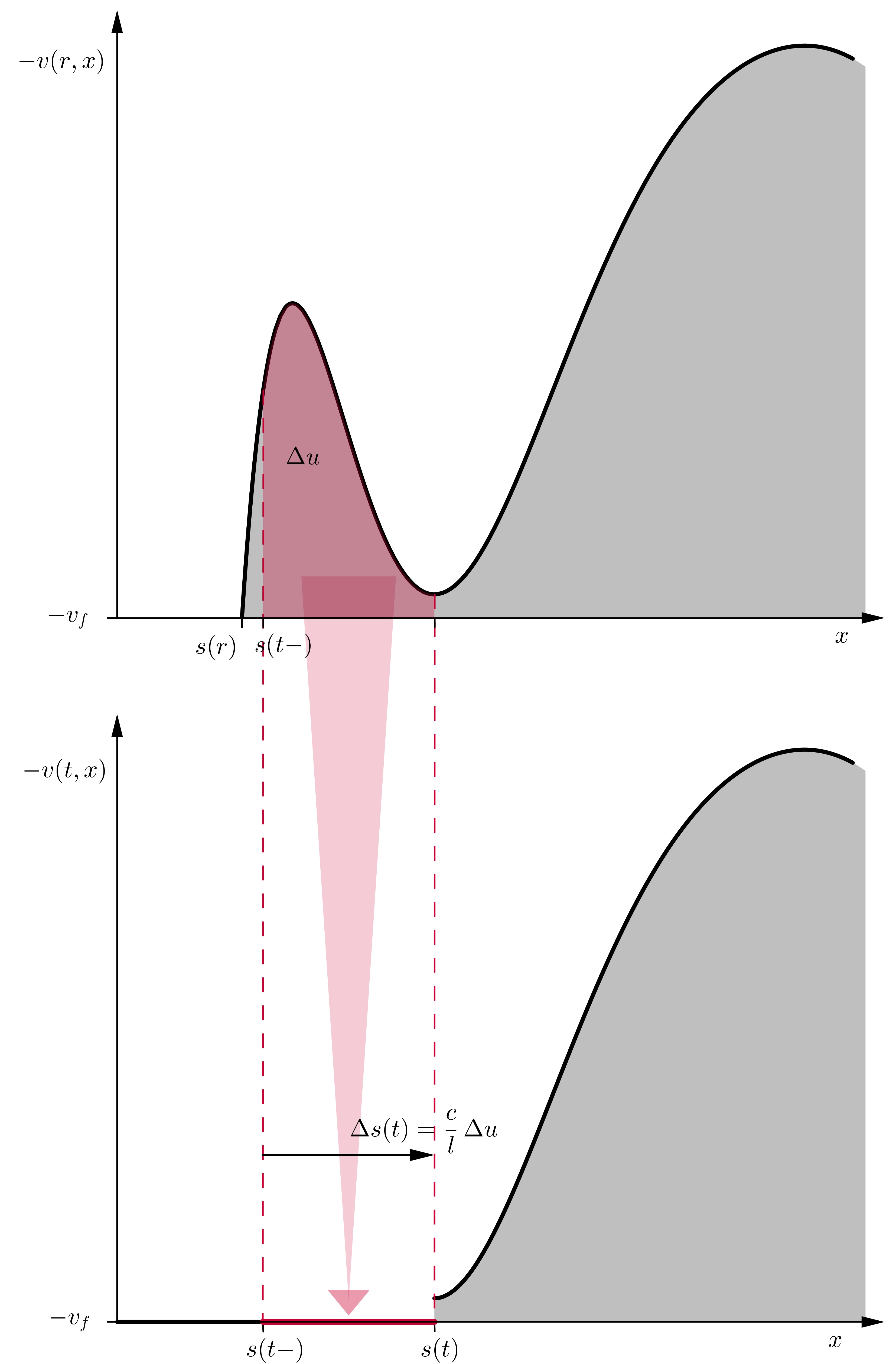}\vspace{-6pt}
	\caption{Illustration of a jump discontinuity in the temperature profile at time $t$, along with the corresponding instantaneous advance of the freezing front according to \eqref{eq:simplified_energy_balance}. In the upper picture, $s(r)\uparrow s(t-)$ as $r\uparrow t$.}\label{fig:jump_illustration}\vspace{-5pt}
\end{figure}

\subsection{C\`adl\`ag weak formulation}

Let $\textbf{D}^._\mathbb{R}$ be the space of real-valued c\`adl\`ag paths on $[0,\infty)$ and we write $\textbf{D}_{L^1(\mathbb{R})}$ for the space of $L^1(\mathbb{R})$-valued c\`adl\`ag paths on $[0,\infty)$ in the sense that
$f \in \textbf{D}_{L^1(\mathbb{R})}$ if and only if $f(t,\cdot) \in L^1(\mathbb{R})$ for all $t\geq0$ and $t\mapsto \langle f(t,\cdot ) , \phi \rangle$ is in $\textbf{D}_\mathbb{R}$ for all $\phi \in C_b(\mathbb{R})$. Furthermore, we write $\textbf{D}^\uparrow_\mathbb{H}$ for the space of increasing (i.e., non-decreasing) continuous paths in $\textbf{D}^._\mathbb{R}$ with values in $\mathbb{H}$. Given the initial data $s_0\in \mathbb{H}$ and $v_0 \in L^1 \cap L^\infty(\mathbb{R})$ with $\mathrm{supp}(v_0)\subseteq [s_0,\infty)$, we then define an admissible class
\begin{align*}
	\textbf{D}(v_0,s_0):= \Bigl\{  &(v,s) \in \textbf{D}_{L^1(\mathbb{R})}\times \textbf{D}^\uparrow_{\mathbb{H}}  :  s(0)=s_0, \,v(0,\cdot)=v_0 \\[-4pt] &\;\;\;\mathrm{supp}(v(t,\cdot))\subseteq [s(t),\infty), \; \Vert v(t,\cdot)\Vert_{L^\infty} \leq \Vert v_0 \Vert_{L^\infty} \;\text{for all } t \geq 0      \Bigr\},
\end{align*}
and we let $\mathfrak{D}(v_0,s_0)$ be the space of $\mathcal{F}_t$-adapted processes with paths in 	$\textbf{D}(v_0,s_0)$, where $(\mathcal{F}_t)_{t\geq 0}$ is the complete, right-continuous filtration generated by the Brownian motion $W$.

\begin{defn}[C\`adl\`ag weak solutions]\label{def:weak_jumps} Let $s_0 \in \mathbb{H}$ and $v_0\in L^1 \!\cap\! L^\infty(\mathbb{R})$ with $\mathrm{supp}(v_0)\subseteq[s_0,\infty)$ and $v_0 \leq 0$ (a.e.). We say that $(v,s) \in \mathfrak{D}(v_0,s_0)$ is a (global) c\`adl\`ag weak solution to the supercooled Stefan problem \eqref{eq:classical_noise} if it holds, almost surely, that
	\begin{align}\label{eq:weak_jump}
		\int_{s(t)}^{\infty}&v(t,x) \phi(t,x) \mathrm{d}x -  \int_{s_0}^{\infty}\!\!v_0(x) \phi(0,x) \mathrm{d}x =  \int_0^t \int_{s(r)}^{\infty }v(r,x)\partial_r \phi(r,x) \mathrm{d}x \mathrm{d}r  \nonumber \\
		&\!\! +\kappa\!\int_0^t \int_{s(r)}^{\infty } v(r,x) \partial_{xx }\phi(r,x) \mathrm{d}x\mathrm{d}r+\theta\!\int_0^t\! \int_{s(r)}^{\infty } \!v(r,x) \partial_x \phi(r,x) \mathrm{d}x\mathrm{d} W_{r}  \nonumber \\
		&\!\!+\lambda\kappa\! \int_0^t \!\!\phi(r, s(r-) )\mathrm{d} s(r)+\sum_{0<r\leq t} 	\int_{s(r-)}^{s(r)}\!{v(r-,x)} \bigl(\phi(r,s(r-))-\phi(r,x)\bigr)  \mathrm{d}x,
	\end{align}
	for all $t>0$ and all $\phi \in C_b^{1,2}([0,\infty)\times \mathbb{R})$.
\end{defn}

The domain of this problem is illustrated in Figure \ref{fig:discont_domain} below. As depicted in Figure \ref{fig:jump_illustration}, the emergence of a temperature discontinuity at time $t$ results in a violation of the isothermal boundary condition $v(t,s(t)+)=0$ at the jump time. However, by the c\`adl\`ag property, there can be at most countably many such times for each path. Since the weak formulation is integrated in time, it is thus consistent for \eqref{eq:weak_jump} to implicitly enforce $v(r,s(r)+)=0$ at almost all times, as per the derivation of Definition \ref{def:weak_cont} above.

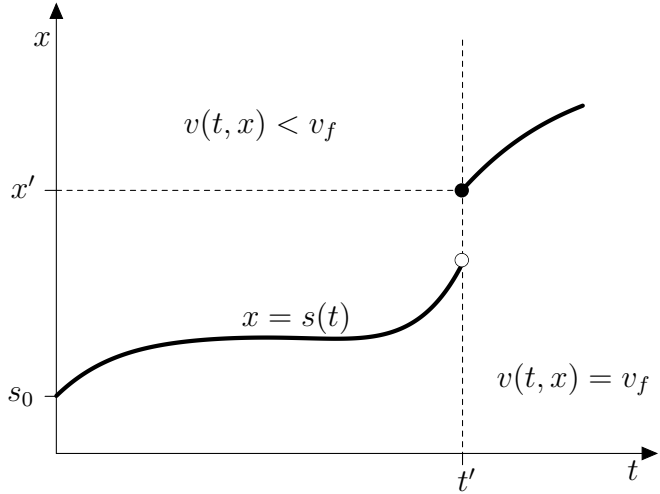
\begin{figure}
	\centering
	\begin{tikzpicture}[line cap=round,line join=round,>=triangle 45,x=1.0cm,y=1.0cm]
		\clip(-1.04,-0.81) rectangle (11.64,6.56);
		\draw [->] (0,0) -- (0,6);
		\draw [->] (0,0) -- (8,0);
		\draw (5.40,0)-- (5.40,-0.15);
		\draw (0,0.76)-- (-0.15,0.76);
		\draw (-0.45,5.75) node[anchor=north west] {$ x $};
		\draw (7.45,0.05) node[anchor=north west] {$ t $};
		\draw (5.19,-0.05) node[anchor=north west] {$ t'$};
		\draw (-0.8,1.02) node[anchor=north west] {$ s_0 $};
		\draw [dash pattern=on 2pt off 2pt] (5.40,5.5)-- (5.40,0);
		\draw[line width=1.6pt, smooth,samples=100,domain=5.42:7.0] plot(\x,{0.04*(\x)^3-0.97*(\x)^2+8.09*(\x)-18.19});
		\draw [dash pattern=on 2pt off 2pt] (0,3.5)-- (5.42,3.5);
		\draw (0,3.5)-- (-0.15,3.5);
		\draw (-0.75,3.85) node[anchor=north west] {$ x' $};
		\draw[line width=1.6pt, smooth,samples=100,domain=0.0:5.37] plot(\x,{0.5*(2.31*(\x)-0.2-1.08*((\x)-0.2)^2+0.1*((\x)-0.2)^3-0.77*sin(((\x)-0.2)*180/pi)-0.41*cos(((\x)-0.2)*180/pi)+0.03*2.718281828^((\x)-0.2)+2)});
		\draw (1.55,4.73) node[anchor=north west] {$ v(t,x) < v_f $};
		\draw (5.71,1.35) node[anchor=north west] {$ v(t,x) = v_f $};
		\draw (2.32,2.15) node[anchor=north west] {$ x = s(t) $};
		\begin{scriptsize}
			\draw [color=black] (5.39,2.572) circle (2.5pt);
			\fill [color=white] (5.39,2.572) circle (2.4pt);
			\fill [color=black] (5.39,3.5) circle (2.7pt);
		\end{scriptsize}
	\end{tikzpicture}\vspace{-10pt}
	\caption{Domain of the supercooled Stefan problem for a single jump discontinuity at time $t^\prime$ with $s(t^\prime)=x^\prime$. In fact, $s(t)$ is only differentiable a.e., but the smooth picture is meant to illustrate the divergence in the rate of increase towards $t^\prime$. We also note that $t^\prime$ could instead be an accumulation point of small subsequent jumps.}\label{fig:discont_domain}
\end{figure}

As for a continuous weak solution, taking $\phi \equiv 1$ in \eqref{eq:weak_jump} returns the criterion \eqref{eq:energy_balance}. Furthermore, if we fix $r>0$ and take $\phi\in C_b^\infty(\mathbb{R})$ such that $\phi(s(r-))=0$, then \eqref{eq:weak_jump} gives
\begin{equation}\label{eq:jump_weak_phi}
	\int_{s(r)}^\infty  \!v(r,x)\phi(x)\mathrm{d}x -  \int_{s(r-)}^{\infty} \hspace{-0.5pt} \!v(r-,x) \phi(x)\mathrm{d}x = - \int_{s(r-)}^{s(r)}\!v(r-,x) \phi(x) \mathrm{d}x.
\end{equation}
Combined with \eqref{eq:energy_balance}, this implies the expression \eqref{eq:simplified_energy_balance}, by letting $\phi(x)\rightarrow 1$ for $x>s(r-)$ and applying dominated convergence. More generally, since \eqref{eq:jump_weak_phi} holds for any $\phi\in C_b^\infty(\mathbb{R})$, we deduce that $v(r,x)=v(r-,x)$ for almost every $x>s(r)$, as desired. Note that this together with \eqref{eq:energy_balance} also gives us \eqref{eq:simplified_energy_balance}.

\subsection{Instability under an external heat transfer}\label{sect:cascades}

The constraint \eqref{eq:simplified_energy_balance} provides an implicit relation for $s(t)$ in terms of only the left-limiting temperature profile $v(t-,\cdot)$ and the corresponding interface $s(t-)$. However, this does not specify the value of $s(t)$ uniquely and nor does it address what phenomenon a discontinuity represents, beyond the fact that it allows for a global notion of solution that conserves heat. As we now explain, there is indeed a natural interpretation, whereby a jump in the freezing front captures the (idealised) resolution of an infinitesimal instability with respect to a vanishing heat transfer from outside the otherwise isolated system.
This, in turn, yields a selection principle for \eqref{eq:simplified_energy_balance} which, as we shall see in Theorem \ref{thm:minimality} below, will be realised by a particular c\`adl\`ag weak solution.

Fixing $t>0$, we let the current state of the system be described by its left-limiting temperature profile $v(t-,\cdot)$ and interface $s(t-)$. We then apply a small external amount of heat to the system, raising the temperature to $0$ on an $\varepsilon$-neighbourhood of the interface and hence initiating freezing so that the front $s$ advances to $s(t-)+\varepsilon$, for a small $\varepsilon>0$. Now, the freezing on $(s(t-),s(t-)+\varepsilon)$ releases an amount of latent heat $\varrho \ell \varepsilon $. Since it is an isolated system (aside from the initial heat transfer), it follows as in \eqref{eq:energy_bal_first_princip} that the temperature goes to zero in a right-neighbourhood of the interface, causing a further advance of the freezing front to $s^{1}(t;\varepsilon)>s(t-)+\varepsilon$ given by
\begin{equation}\label{eq:inductive_freezing_step0}
	\varepsilon = - \frac{1}{\lambda \kappa} \int_{s(t-)+\varepsilon}^{s^{1}(t;\varepsilon)} v(t-,x)\mathrm{d}x.
\end{equation}
But then this freezing yields an additional latent heat release of $\varrho \ell (s^{1}(t;\varepsilon)- s^0(t;\varepsilon))$, where $s^0(t;\varepsilon):=s(t-)+\varepsilon$, heating the liquid to zero near the interface and further advancing the freezing front to $s^{2}(t;\varepsilon)>s^{1}(t;\varepsilon)$ given by 
\[
s^{1}(t;\varepsilon)- s^0(t;\varepsilon) = -\frac{1}{\lambda \kappa} \int_{s^{1}(t;\varepsilon)}^{s^{2}(t;\varepsilon)} v(t-,x)\mathrm{d}x.
\]
Continuing in this way, we obtain the recursion
\begin{equation}\label{eq:inductive_freezing}
	s^{n}(t;\varepsilon)- s^{n-1}(t;\varepsilon) = -\frac{1}{\lambda \kappa} \int_{s^{n}(t;\varepsilon)}^{s^{n+1}(t;\varepsilon)} v(t-,x)\mathrm{d}x, \quad n\geq 1,
\end{equation}
which we note determines $s^{n+1}(t;\varepsilon)>s^{n}(t;\varepsilon)$ uniquely, given $s^{n}(t;\varepsilon)$ and $s^{n-1}(t;\varepsilon)$, since $v(t-,x)<0$ for $x>s(t-)$. By telescoping, \eqref{eq:inductive_freezing_step0} and  \eqref{eq:inductive_freezing} give
\[
s^{n}(t;\varepsilon) = s(t-)  - \frac{1}{\lambda \kappa} \int_{s(t-) + \varepsilon}^{s^{n+1}(t;\varepsilon)} v(t-,x)\mathrm{d}x,
\]
so the final position $s(t;\varepsilon)$ of the interface satisfies
\begin{equation}\label{eq:limit_freeze_front}
	s(t;\varepsilon)   = s(t-) - \frac{1}{\lambda \kappa} \int_{s(t-) + \varepsilon}^{s(t;\varepsilon)} v(t-,x)\mathrm{d}x,\quad s(t;\varepsilon):=\lim_{n\rightarrow \infty} s^{n}(t;\varepsilon).
\end{equation}
Observe that in \eqref{eq:limit_freeze_front}, the temperature increase on $(s(t-),s(t-)+\varepsilon)$ is unaccounted for, as it arose from an external heat transfer. By sending the magnitude of this heat transfer to zero, we bring back conservation of heat: indeed, applying dominated convergence in \eqref{eq:limit_freeze_front}, the limiting outcome $\lim_{\varepsilon \downarrow 0} s(t; \varepsilon)$ obeys the constraint \eqref{eq:simplified_energy_balance} and hence is a valid value for $s(t)$.

\begin{rem}[Physical jumps] Arguing analogously to \cite[Proposition 2.4]{hambly_ledger_sojmark_2018}, we get
	\begin{equation}\label{eq:min_jump}
		\lim_{\varepsilon \downarrow 0} s(t; \varepsilon) = s(t-) + \inf 
		\, \Bigl\{ y>0 : - \frac{1}{\lambda \kappa} \int_{s(t-)}^{s(t-)+y} \!v(t-,x)\mathrm{d}x < y \Bigr\}
	\end{equation}
	which we recognise as a variant of the \emph{physical} jump condition from \cite{DIRT_SPA} that has been used in recent probabilistic approaches to the supercooled Stefan problem \cite{cuchiero, DNS,  hambly_ledger_sojmark_2018, nadtochiy_shkolnikov_2017}.
\end{rem}

\section{Main results}\label{sect:main_results}

As above, we stress that this section is concerned with initial temperature profiles $v_0\leq0$ (a.e.), since the equilibrium freezing temperature is set to $v_f=0$. Next, we can observe that, in order to have solutions in $\mathfrak{D}(v_0,s_0)$, we must enforce that $v_0$ is stable under an infinitesimal external heat transfer in the sense of Section \ref{sect:cascades}. Taking $s(0-):=s_0$ and $v(0-,\cdot):=v_0$ in the definition of $ s(0;\varepsilon)$, it follows from \eqref{eq:min_jump} that what we need is
\begin{equation}\label{eq:ic_constraint}
	- \frac{1}{\lambda \kappa} \int_{s_0}^{s_0+y} \!v_0(x)\mathrm{d}x < y  \quad \text{infinitely often,} \quad \text{as } y \downarrow 0.
\end{equation}
Note that, if this is not satisfied, one can first compute $\tilde{s}_0:=\lim_{\varepsilon \downarrow 0} s(0; \varepsilon)$, shift the initial value of the front to $\tilde{s}_0>s_0$, and then the new initial condition $v_0\vert_{[\tilde{s}_0,\infty)}$ satisfies \eqref{eq:ic_constraint} for $\tilde{s}_0$. When the initial condition satisfies the constraint \eqref{eq:ic_constraint}, we have the following result.

\begin{thm}[Probabilistic c\`adl\`ag solutions]\label{thm:existence_jump}
	Let $s_0\in \mathbb{H}$ and $v_0 \in L^1 \cap L^\infty(\mathbb{R})$ with $\mathrm{supp}(v_0)\subseteq [s_0,\infty)$. If \eqref{eq:ic_constraint} is satisfied, then there exist weak solutions $(v,s)\in \mathfrak{D}(v_0,s_0)$ in the sense of Definition \ref{def:weak_jumps} such that the probabilistic representation \eqref{eq:prob_rep_0}--\eqref{eq:prob_rep} holds globally in time. 
\end{thm}

Next, we address the emergence of blow-ups in the sense that, with strictly positive probability, an instability develops in finite time and the continuous weak formulation breaks down. Of course, this blow-up is not the end of things, as the solution can continue to evolve in a c\`adl\`ag manner according to Definition \ref{def:weak_jumps}.

For any given c\`adl\`ag weak solution $(v,s)\in \mathfrak{D}(v_0,s_0)$, we define
\begin{equation}\label{eq:blow_up_time}
	\varsigma:= \inf \bigl\{  t>0 : v(t,\cdot ) \neq   v(t-,\cdot ) \bigr\} = \inf \bigl\{  t>0 : s(t) \neq   s(t- ) \bigr\}.
\end{equation}
As usual, $\inf \emptyset = +\infty$. The equivalence of the two expressions in \eqref{eq:blow_up_time} is an immediate consequence of the weak formulation in Definition \ref{def:weak_jumps}.

\begin{thm}[Temperature discontinuities]\label{thm:jump_nonzero_prob}
	Suppose $v_0(y)<-\lambda \kappa$ at some point $y$ in $(s_0,\infty)$ and let $v_0$ be right-continuous at that point. Then we have
	\[
	\mathbb{P}( \varsigma < +\infty) > 0
	\]
	for any c\`adl\`ag weak solution $(v,s)\in \mathfrak{D}(v_0,s_0)$ in the sense of Definition \ref{def:weak_jumps}.
\end{thm}

The above confirms that temperature discontinuities are an inherent feature of the supercooled Stefan problem with noise. This is very different from the deterministic problem \eqref{eq:classical_soln}, for which, e.g., \cite[Theorem 1.1]{bayraktar} and \cite[Proposition 2.4]{LS} give simple sufficient criteria guaranteeing global continuity even if the initial temperature profile is supercooled below $-\lambda \kappa$. 

Note that Theorem \ref{thm:jump_nonzero_prob} is a general statement for all c\`adl\`ag weak solutions. Next, we will single out a special---uniquely defined---solution, which exhibits a minimal increase in temperature over time. For this solution, we are able to show that temperature discontinuities are resolved in accordance with the mechanism in Section \ref{sect:cascades}.

\begin{thm}[Minimal temperature increase]\label{thm:minimality}
	Let $s_0\in \mathbb{H}$ and $v_0 \in L^1 \cap L^\infty(\mathbb{R})$ with $\mathrm{supp}(v_0)\subseteq [s_0,\infty)$ and suppose \eqref{eq:ic_constraint} holds. Then there exists a probabilistic c\`adl\`ag solution  $(v,s)\in \mathfrak{D}(v_0,s_0)$ of minimal total temperature increase, i.e, for any other probabilistic c\`adl\`ag solution $(\tilde{v},\tilde{s})$ from Theorem \ref{thm:existence_jump}, we have
	\[
	\int_{s_0}^\infty v_0(x)\mathrm{d}x \leq	\int_{s(t)}^\infty v(t,x)\mathrm{d}x \leq \int_{\tilde{s}(t)}^\infty \tilde{v}(t,x)\mathrm{d}x , \quad \text{for all}\quad t\geq0,
	\]
	almost surely. Furthermore, at every $\mathcal{F}_t$-stopping time $\tau$, this solution satisfies
	\begin{equation}\label{eq:equality_jumps}
		s(\tau) = \lim_{\varepsilon \downarrow 0}  s(\tau;\varepsilon)
	\end{equation}
	on $\{\tau \in (0, \infty) \}$ almost surely.
\end{thm}

For the deterministic Stefan problem with $\theta =0$, \eqref{eq:equality_jumps} says that, for all $t > 0$, we have $s(t)= \lim_{\varepsilon \downarrow 0}  s(t;\varepsilon)$. Phrased in terms of \eqref{eq:min_jump} and the probabilistic representation \eqref{eq:prob_rep}, a result of this form was recently established in \cite{cuchiero} under slightly more restrictive assumptions. We note that \cite{cuchiero} follows a completely different approach that does not carry over to our setting and which involves analysis of a related particle system (see Remark \ref{rem:physical_minimal} for further details).

When $\theta \neq 0$, it is not immediately clear if \eqref{eq:equality_jumps} implies a corresponding pathwise statement with probability 1. Returning to the deterministic problem (in the form of its probabilistic representation), it was observed in \cite{nadtochiy_shkolnikov_2020} that, perturbing $X$ in \eqref{eq:prob_rep} by a path $f$ satisfying $\liminf_{t\downarrow 0} f(t)/\sqrt{t}>0$, it is possible to have a solution with $s(0)=s_0$ despite violating \eqref{eq:ic_constraint}. That is, a sufficiently fast transport away from the freezing front may be able to compensate for an instability of the temperature profile.

For any given path $t\mapsto W_t(\omega)$ (in an almost sure set of paths), the square root laws for Brownian motion (see \cite{Davis}) yield a dense set of times $t$ at which there is an $\varepsilon>0$ and $c\in(0,1)$ so that $W_{t+r}(\omega)-W_{t}(\omega)\geq c \sqrt{r}$ for all $r\in[0,\varepsilon]$. Thus, one could worry if there are times at which the solution of minimal temperature increase may satisfy $s(t)(\omega)=s(t-)(\omega)$ despite $\lim_{\varepsilon \downarrow 0} s(t;\varepsilon)(\omega) > s(t-)(\omega)$, meaning that an instability does not dictate a discontinuity. We are able to rule this out, with probability 1, by utilising \eqref{eq:equality_jumps} and right-continuity of the solution. As we can base our arguments on Theorem \ref{thm:minimality}, we give the proof already here.

\begin{thm}[Pathwise characterisation of \eqref{eq:equality_jumps}]\label{thm:traject_jump} Assume \eqref{eq:ic_constraint} holds and let $(v,s) \in \mathfrak{D}(v_0,s_0)$ be the solution of minimal temperature increase from Theorem \ref{thm:minimality}. Then we have
	\begin{equation}\label{eq:fragility_lower_bound}
		s(t) = \lim_{\varepsilon \downarrow 0} s(t;\varepsilon), \quad \text{for all}\quad  t\geq0,
	\end{equation}
	with probability 1.
\end{thm}
\begin{proof}
	Since $s$ is c\`adl\`ag and $\mathcal{F}_t$-adapted, we can write
	\[
	\{ (t,\omega )\in (0,\infty)\times \Omega : s(t)(\omega)\neq s(t-)(\omega)\} = \bigcup_{n\geq 1} \{ (t,\omega )\in (0,\infty)\times \Omega : \tau_n(\omega)=t \}
	\]
	for a countable family of strictly positive $\mathcal{F}_t$-stopping times $\tau_n$, see e.g.~\cite[Theorem 3.32]{he_wang_yan}. For each $\tau_n$, we can apply Theorem \ref{thm:minimality} to find an event $\Omega_n \in \mathcal{F}_\infty$ with $\mathbb{P}(\Omega_n)=1$ on which \eqref{eq:equality_jumps} holds. Defining $\Omega_\star \in \mathcal{F}_\infty$ as the countable intersection of these events, it holds for all $\omega \in \Omega_\star $ that
	\begin{equation}\label{eq:traject_equal}
		s(t)(\omega) = \lim_{\varepsilon \downarrow 0} s(t;\varepsilon)(\omega)
	\end{equation}
	for every $t >0$ with $s(t)(\omega)\neq  s(t-)(\omega)$. Throughout the proof, we restrict to $\Omega_\star$.
	
	For simplicity of notation, we will write $\nu(t,\mathrm{d}x)=-v(t,x)\mathrm{d}x$. In Section \ref{sect:cascades}, setting aside the precise interpretation in terms of an external heat transfer, we can analogously define	$\tilde{s}^1(t;\varepsilon):=s(t-)+\frac{1}{\lambda \kappa}\nu(t-,[s(t-),s(t-)+\varepsilon])$ and then $\tilde{s}^{n+1}(t;\varepsilon):=s(t-)+\frac{1}{\lambda \kappa}\nu(t-,[s(t-),\tilde{s}^{n}(t;\varepsilon)+\varepsilon])$. Sending $n\rightarrow \infty$ and subsequently $\varepsilon \downarrow 0$, we again arrive at $\lim_{\varepsilon \downarrow 0}s(t;\varepsilon)$. It is immediate that $\tilde{s}^1$ is progressively measurable. From $(v,s)\in\mathfrak{D}(v_0,s_0)$, we can deduce that $t\mapsto \nu(t,[s(t),s(t)+y])$ is c\`adl\`ag for all $y\geq0$ and $y\mapsto \nu(t,[s(t),s(t)+y])(\omega)$ is continuous for all $t\geq 0$ and $\omega\in\Omega_\star$. By induction, it is therefore straightforward to confirm that each $\tilde{s}^{n}(t;\varepsilon)$ is progressively measurable, and so also $\lim_{\varepsilon \downarrow 0}s(t;\varepsilon)$ is progressively measurable. Now, by \eqref{eq:traject_equal}, we have $s(t)=s(t-)$ and $\lim_{\varepsilon\downarrow 0} s(t;\varepsilon)>s(t-)$ if and only if $\lim_{\varepsilon\downarrow 0} s(t;\varepsilon) > s(t)$, so we are interested in
	\[\tau:= \inf\bigl\{ t>0 : \lim_{\varepsilon\downarrow 0} s(t;\varepsilon) > s(t)  \bigr\}
	\]
	which defines a stopping time by the above. If we can show that $\mathbb{P}(\tau < \infty)=0$, then \eqref{eq:fragility_lower_bound} follows from \eqref{eq:traject_equal}. To this end, we write $ \{ \tau < \infty\}=\cup_{k\geq 1} 	\{ \tau_k < \infty\}$,
	where also each $\tau_k := \inf \{ t>0 : \lim_{\varepsilon\downarrow 0}  s(t;\varepsilon) \geq  s(t) + 1/k \}$ is a stopping time. Suppose $\mathbb{P}(\tau_k<\infty)>0$ for some $k\geq 1$. Then, for arbitrary $\omega \in \{\tau_k<\infty\}$, we have $\lim_{\varepsilon \downarrow 0}s(\tau_k;\varepsilon)(\omega) \geq  s(\tau_k)(\omega) + 1/k$ or $\lim_{\varepsilon \downarrow 0}s(t_n;\varepsilon)(\omega) \geq  s(t_n)(\omega) + 1/k$ for
	some sequence $t_n\downarrow \tau_k(\omega)$. By \eqref{eq:traject_equal} we must have $s(t_n)=s(t_n-)$, as we otherwise have a contradiction. Thus, also $v(t_n-,\cdot)=v(t_n,\cdot)$, and so the definition of $s(t_n;\varepsilon)$ gives $\nu(t_n,[s(t_n),s(t_n)+y])(\omega)\geq \lambda\kappa y$ for all $y\in(0, 1/k]$, for every $n\geq 1$.
	Using the form of $\nu$ and its right-continuity (as well as that of $s$), we can therefore deduce that
	\begin{equation}\label{eq:bound_nu_contradict}
		\nu(\tau_k(\omega),[s(\tau_k(\omega)),s(\tau_k(\omega))+y])(\omega)\geq \lambda\kappa y\quad \text{for all} \quad y\in(0, 1/k].
	\end{equation}
	When $\tau_k=0$, since $s(0)=s_0$, we get $\nu_0([s_0,s_0+y])(\omega)\geq \lambda\kappa y$ for all $y\in(0, 1/k]$, which contradicts \eqref{eq:ic_constraint}, so we may assume $\mathbb{P}(\tau_k\in(0,\infty))>0$. From \eqref{eq:jump_weak_phi}, we have
	\begin{align}\label{eq:nu_t_and_nu_t_minus}
		\nu(\tau_k,[s(\tau_k),s(\tau_k)+y]) &=  \nu(\tau_k-,[s(\tau_k-),s(\tau_k)+y]) -    \nu(\tau_k-,[s(\tau_k-),s(\tau_k)]) \nonumber \\
		&= \nu(\tau_k-,[s(\tau_k-),s(\tau_k)+y]) - \lambda \kappa \Delta s(\tau_k) \nonumber \\
		&=  \nu(\tau_k-,[s(\tau_k-),s(\tau_k-)+z]) - \lambda \kappa \Delta s(\tau_k)
	\end{align}
	with $z=\Delta s(\tau_k)+y$. It follows from \eqref{eq:nu_t_and_nu_t_minus} and \eqref{eq:bound_nu_contradict} that
	\begin{equation}\label{eq:jump-tau_k}
		\nu(\tau_k-,[ s(\tau_k-), s(\tau_k-)+z]) \geq \lambda \kappa z\quad 	\text{for all}\quad  z\in[\Delta s(\tau_k),\Delta s(\tau_k)+1/k].
	\end{equation}
	Applying Theorem \ref{thm:minimality} to $\tau_k$ (on the event $\tau_k\in(0,\infty)$), we have $s(\tau_k)=\lim_{\varepsilon \downarrow 0}s(\tau_k;\varepsilon)$, which, by definition of $s(t;\varepsilon)$, gives that \eqref{eq:jump-tau_k} also holds for all $z\in[0,\Delta s(\tau_k)]$. Again by definition of $s(t;\varepsilon)$, this implies
	$\lim_{\varepsilon \downarrow 0}s(\tau_k;\varepsilon)\geq s(\tau_k-)+(\Delta s(\tau_k) + 1/k) = s(\tau_k)+1/k$, which is a contradiction. Thus, we have $\mathbb{P}(\tau_k<\infty)=0$ for all $k\geq1$, and consequently $\mathbb{P}(\tau <\infty)=0$, as desired.
\end{proof}

While the above analysis is concerned with discontinuities, we end the section with two complementary results on continuity. The first is a simple uniqueness result that may be readily deduced by combining Theorem \ref{thm:prob_rep_cont} and Theorem \ref{thm:existence_jump} with an existing result.

\begin{thm}[Unique continuous solution]\label{thm:no_jump} Let $\operatorname*{ess\,inf}_{x\in (s_0,\infty)}v_0(x) > -\lambda\kappa$. Then there exists a continuous weak solution $(v,s)\in \mathfrak{C}(v_0,s_0)$ and this solution is unique in the larger class of c\`adl\`ag weak solutions given by Definition \ref{def:weak_jumps}.
\end{thm}
\begin{proof}
	Note that \eqref{eq:ic_constraint} holds, so there exists a c\`adl\`ag weak solution by Theorem \ref{thm:existence_jump}. As observed in \eqref{eq:jump_weak_phi}, any c\`adl\`ag weak solution must satisfy \eqref{eq:simplified_energy_balance}, but the assumption $\operatorname*{ess\,inf}_{x\in (s_0,\infty)}v_0(x) > -\lambda\kappa$ implies that $s(t)=s(t-)$ is the only solution to \eqref{eq:simplified_energy_balance} at every $t>0$, so it follows from this and \eqref{eq:weak_jump} that any c\`adl\`ag weak solution is in $\mathfrak{C}(v_0,s_0)$. Thus, Theorem \ref{thm:prob_rep_cont} applies to every c\`adl\`ag weak solution, and we have $\Vert v_0 \Vert_{L^\infty(\mathbb{R})} < \lambda\kappa $ (recall $v_0\leq 0$ a.e.), so uniqueness follows in complete analogy with \cite[Theorem 2.2]{LS_bulletin}.
\end{proof}

When the condition in Theorem \ref{thm:no_jump} is violated, we know from Theorem \ref{thm:jump_nonzero_prob} that we need to deal with the jump discontinuities in Definition \ref{def:weak_jumps}, and it is unclear whether or not there is a unique solution in general. However, as long as \eqref{eq:ic_constraint} holds, Theorem \ref{thm:minimality} confirms that we have the unique solution of minimal temperature increase $(v,s)\in \mathfrak{D}(v_0,s_0)$. Provided $v_0$ does not dip below the critical value $-\lambda\kappa$ near the initial position of the freezing front $s_0$, the next result shows that this solution $(v,s)$ will evolve continuously for some small time and that it will do so for all time with strictly positive probability.

\begin{thm}[Initial and global continuity]\label{prop:initial_global_cont}
	Let \eqref{eq:ic_constraint} hold with $v_0(x) \geq - \lambda \kappa$ a.e.~in a right-neighbourhood of $s_0$. Then the solution of minimal temperature increase $(v,s)\in \mathfrak{D}(v_0,s_0)$ from Theorem \ref{thm:minimality} satisfies 
	\[
	\mathbb{P}(\varsigma >0)=1\quad \text{and}\quad \mathbb{P}( \varsigma = +\infty)>0.
	\]
\end{thm}

\section{The associated McKean--Vlasov problem}\label{sect:McKean}

As in Section \ref{sect:temp_discont}, we shall work with the space $\mathbf{D}^{\uparrow}_{\mathbb{H}}$ of increasing (i.e., non-decreasing) right-continuous paths $f:[0,\infty)\rightarrow \mathbb{H}$, where we recall that $\mathbb{H}=[0,\infty)$ denotes the positive half-line. Moreover, we shall work on a suitable probability space $(\Omega , \mathcal{F},\mathbb{P})$ that supports two independent Brownian motions $B$ and $W$ as well as a finite random measure $\nu_0$ on $\mathbb{H}$ and a non-negative random variable $s_0$ both of which are independent of the Brownian motions. We let $(\mathcal{F}_t)_{t\geq0}$ denote the augmented right-continuous filtration generated by $W$, where $\mathcal{F}_0$ is enlarged such that $\nu_0$ and $s_0$ are $\mathcal{F}_0$-measurable. Throughout, it is understood that $\nu_0$ is supported on $[s_0,\infty)$.

Let $\mathfrak{D}^{\uparrow}_{\mathbb{H}}$ denote the space of $\mathcal{F}_t$-adapted processes with paths in $\mathbf{D}^{\uparrow}_{\mathbb{H}}$. We are then interested in solutions $s\in \mathfrak{D}^{\uparrow}_{\mathbb{H}}$ to
\begin{equation}\label{eq:MV}
	\left\{
	\begin{aligned}
		X^y_t &= y + \textstyle\sqrt{2 \kappa-\theta^2} B_t + \theta W_t\\[3pt]
		\tau^y &= \inf\{ t \geq 0 : X^y_t \leq s({\color{black}t}) \} \\[1pt]
		s({\color{black}t}) &= s_{\color{black}0} + \frac{1}{\lambda \kappa }\int_{s_0}^\infty\mathbb{P}(\tau^y\leq t \mid \mathcal{F}_t)\nu_0(\mathrm{d}y).
	\end{aligned} \right.
\end{equation}
Note that, in the setting of Theorem \ref{thm:prob_rep_cont}, the initial position $s_0$ and the initial temperature profile $v_0$ were both deterministic. Here, we allow them to be random, as this is the natural setting when we shall later consider the restarting of the problem. Finally, we stress that the initial measure $\nu_0$ in \eqref{eq:MV} would correspond to $-v_0(y)\mathrm{d}y$ in the previous sections.

Similarly to the map $\Gamma$ considered in \cite[Equation (1.6)]{hambly_ledger_sojmark_2018}, we shall study this problem through fixed points of the mapping $\Gamma: \mathfrak{D}^{\uparrow}_{\mathbb{H}} \rightarrow \mathfrak{D}^{\uparrow}_{\mathbb{H}}$ defined by
\[
\Gamma[s](t):= s_0 + \frac{1}{\lambda \kappa }\int_{s_0}^\infty\mathbb{P}(\tau^{y,s} \leq t \mid \mathcal{F}_\infty)\nu_0(\mathrm{d}y),\quad \text{for } t\geq0,
\]
where
\begin{equation}\label{eq:Gamma_map}
	\left\{
	\begin{aligned}
		X^y_t &= y + \textstyle\sqrt{2 \kappa-\theta^2} B_t + \theta W_t\\[3pt]
		\tau^{y,s} &= \inf\{ t \geq 0 : X^y_t \leq s({\color{black}t}) \} .
	\end{aligned} \right. 
\end{equation}

Since $s\in \mathfrak{D}^{\uparrow}_{\mathbb{H}}$ is $\mathcal{F}_t$-adapted, and since $W$ and $B$ are independent Brownian motions with respect to $\mathcal{F}$, we have $\mathbb{P}(\tau^{y,s} \leq t \mid \mathcal{F}_\infty)=\mathbb{P}(\tau^{y,s} \leq t \mid \mathcal{F}_t)$ almost surely, for every $t\geq 0$. In particular, $\Gamma[s]$ is $\mathcal{F}$-adapted as the filtration is complete. Moreover, using completeness and right-continuity of the filtration, there is a version of $\Gamma[s]$ with paths in $\mathbf{D}^{\uparrow}_{\mathbb{H}}$. Taking this as the definition, we have $\Gamma[s]\in \mathfrak{D}^{\uparrow}_{\mathbb{H}}$, as required.

The above confirms that any fixed point $s\in \mathfrak{D}^{\uparrow}_{\mathbb{H}}$ of $\Gamma$ is a solution to \eqref{eq:MV}. Conversely, any solution $s\in \mathfrak{D}^{\uparrow}_{\mathbb{H}}$ of \eqref{eq:MV} yields a fixed point of $\Gamma$. Indeed, as above, $\mathbb{P}(\tau^{y,s} \leq t \mid \mathcal{F}_\infty)=\mathbb{P}(\tau^{y,s} \leq t \mid \mathcal{F}_t)$ almost surely for every $t\geq 0$. Thus, $s(t)=\Gamma[s](t)$ almost surely for all $t\geq 0$. Both are right-continuous, so they are indistinguishable, and hence $s=\Gamma[s]$ in $\mathfrak{D}^{\uparrow}_{\mathbb{H}}$. Now, for any pair $s, \tilde{s} \in\mathfrak{D}^{\uparrow}_{\mathbb{H}} $, we shall write $s \leq \tilde{s}$ if and only if
\begin{equation}\label{eq:ordering}
	\mathbb{P}(s(r) \leq \tilde{s}(r)\;\text{for all } r\in \mathbb{Q}_+)=1,
\end{equation}
where $\mathbb{Q}_+:=\mathbb{Q} \cap [0,\infty)$.  This is of course equivalent to $\mathbb{P}(s(t) \leq \tilde{s}(t)\;\text{for all } t\geq 0)=1$ by right-continuity. With this ordering, we obtain the following result.

\begin{thm}[Solution lattice]\label{thm:min_soln}
	For any ($\mathcal{F}_0$-measurable) initial conditions $\nu_0$ and $s_0$, the set of solutions $s\in \mathfrak{D}^{\uparrow}_{\mathbb{H}} $ to \eqref{eq:MV}
	forms a lattice with respect to the partial order \eqref{eq:ordering}. In particular, there is a minimal and a maximal solution.
\end{thm}
\begin{proof}
	Clearly, $s \leq \tilde{s}$ implies $\Gamma[s]\leq \Gamma[\tilde{s}]$ in $\mathfrak{D}^{\uparrow}_{\mathbb{H}}$. By Lemma \ref{lem:lattic} below and the above observations, the claim follows from Tarski's fixed point theorem.
\end{proof}

\begin{lem}[Completeness]\label{lem:lattic}Consider the set of processes in $\mathfrak{D}^{\uparrow}_{\mathbb{H}}$ that are upper bounded by the random variable $M:=s_0+\nu_0([s_0,\infty))$. This defines a complete lattice with respect to the partial order \eqref{eq:ordering}.
\end{lem}
\begin{proof}
	Fix an arbitrary set $S \subseteq \mathfrak{D}^{\uparrow}_{\mathbb{H}}$. 	We shall construct an element $\check{s}\in \mathfrak{D}^{\uparrow}_{\mathbb{H}}$ so that the join of $S$ can be defined by $\bigvee S := \check{s}$. The meet $\bigwedge S := \hat{s} $ can be defined for an element $\hat{s} $ constructed analogously. For each time $t\geq 0$, we have that $S(t):=\{ s(t) : s\in S \}$ is a collection of $\mathcal{F}_t$-measurable random variables
	which has an $\mathcal{F}_t$-measurable essential supremum $\bar{s}(t)$ (unique up to null sets),  see e.g.~\cite[Proposition 4.1.1]{edgar_sucheston}. Note that each $\bar{s}(t)$ is bounded by $M$, due to our assumption.
	
	For any $r^{\prime} \leq r^{\prime\prime}$ in $\mathbb{Q}_+$, it holds that $\bar{s}(r^{\prime}) \leq \bar{s}(r^{\prime \prime})$ almost surely by definition of the essential supremum (since $s(r^{\prime}) \leq s(r^{\prime \prime})$ for all $s\in S$, so $\bar{s}(r^{\prime \prime})$ is an upper bound of $S(r^\prime)$). Thus, we can find an event $\Omega_*$ (given by countable intersections) with $\mathbb{P}(\Omega_*)=1$ so that $\mathbb{Q}_+\ni r \mapsto \bar{s}(r)(\omega)$ is increasing for all $\omega \in \Omega_*$. Consequently, at every time $t\geq 0$, we may define $\check{s}(t)(\omega):=\lim_{r \downarrow t,\,r\in\mathbb{Q}_+} \bar{s}(r)(\omega)$ for $\omega \in \Omega_*$ and $\check{s}(t)(\omega):=0$ otherwise. By construction, the paths $t \mapsto \check{s}(t)(\omega)$ are in $\mathbf{D}^{\uparrow}_{\mathbb{H}}$ for all $\omega\in \Omega$. Moreover, since the filtration is right-continuous and complete, $\check{s}(t)$ is $\mathcal{F}_t$-measurable, so $\check{s}\in \mathfrak{D}^{\uparrow}_{\mathbb{H}}$ with paths bounded by $M$, as desired.
	
	It remains to confirm that $\check{s}$ defines the join of $S$. To this end, fix any $\tilde{s} \in \mathfrak{D}^{\uparrow}_{\mathbb{H}}$ such that $\tilde{s} \geq s$ for all $s\in S$. Then, for every $r\in\mathbb{Q}_+$, $\tilde{s}(r) \geq s(r)$ almost surely for any $s\in S$, so $\bar{s}(r) \leq \tilde{s}(r)$ almost surely by definition. Taking countable intersections, we get $\mathbb{P}(\bar{s}(r) \leq \tilde{s}(r)\;\text{for all } r\in\mathbb{Q}_+)=1$. Thus, the right-continuity of $\tilde{s}$ and the definition of $\bar{s}$ gives $\mathbb{P}(\check{s}(r) \leq \tilde{s}(r)\;\text{for all } r\in\mathbb{Q}_+)=1$, so $\check{s} \leq \tilde{s}$ in $\mathfrak{D}^{\uparrow}_{\mathbb{H}}$ as required.
\end{proof}

By applying (conditional) dominated convergence and using right-continuity, we can observe that the starting point of the freezing front $s$ must satisfy the constraint
\begin{equation}\label{eq:initial_jump_constraint}
	s(0)=s_0+ \frac{1}{\lambda \kappa }\nu_0([s_0,s(0)]).
\end{equation}
When $\nu_{0}(\{0\})=0$, this is trivially satisfied by $s(0)=s_0$. Nevertheless, depending on the shape of $\nu_0$, there could be multiple values satisfying the constraint, and then a given solution $s\in  \mathfrak{D}^{\uparrow}_{\mathbb{H}} $ could have $s(0)\geq s_0$ with $s(0) \neq s_0$. Regarding the minimal solution, if $s(0) \neq s_0$ (with non-zero probability), then it must be the case that there does \emph{not} exist a solution $s\in  \mathfrak{D}^{\uparrow}_{\mathbb{H}} $ with $s(0)=s_0$ as we would otherwise have a contradiction. Luckily, we can say the following, which lies at the heart of our approach in this paper.

\begin{prop}[Alignment with initial conditions]\label{prop:initial_aligned} Let the initial profile $\nu_0 $ be such that
	\begin{equation}\label{eq:initial_aligned}
		\frac{1}{\lambda \kappa}\nu_0([s_0,s_0 + y]) <  y  \quad \text{infinitely often,} \quad \text{as } y \downarrow 0,
	\end{equation}
	almost surely. Then the minimal solution $s\in  \mathfrak{D}^{\uparrow}_{\mathbb{H}} $ satisfies $s(0)=s_0$ almost surely.
\end{prop}
\begin{proof}
	Define $s^{(0)}:=\Gamma[0]$ and $s^{(n)}:=\Gamma[s^{(n-1)}]$ for all $n\geq 1$. This gives a sequence of increasing continuous processes with $s^{(n)}\geq s^{(n-1)}$ for all $n\geq 1$. Taking  $S=\{s^{(n)}\}_{n\geq1}$ in the proof of Lemma \ref{lem:lattic}, we can simply let $\bar{s}(t):=\lim_{n\rightarrow\infty} s^{(n)}(t)$ for every $t \geq 0$. Then $\bar{s}$ has increasing paths by construction (after setting them to zero on a null set), so we get that, with probability 1, $\check{s} \in \mathfrak{D}^{\uparrow}_{\mathbb{H}}$ from the proof of Lemma \ref{lem:lattic} satisfies $\check{s}(t)=\lim_{n\rightarrow\infty} s^{(n)}(t)$ at every $t>0$ such that $\Delta \check{s}(t)=0$.
	
	By standard properties of Brownian motion and the fact that $\check{s}$ is increasing, it holds almost surely that, for every $y$, each path of $X^y$ satisfies $X^y_{r} < \check{s}(r)$ infinitely often as $r\downarrow \tau^{y,\check{s}}$. Therefore, by the convergence of $s^{(n)}$ to $\check{s}$, we get $ \tau^{y,s^{(n)}}\rightarrow \tau^{y,\check{s}}$ for all $y$ almost surely. Consider now the co-countable set of times $ t >0$ such that  $\mathbb{P}(\Delta \Gamma[\check{s}](t) = 0,\,  \Delta \check{s}(t)=0) =1 $.
	For any such $t$, we have $\nu_0( \{ y : \tau^{y,\check{s}} = t)\})=0$
	almost surely, and hence $ \tau^{y,s^{(n)}}\rightarrow \tau^{y,\check{s}}$ yields $\Gamma[s^{(n)}](t)\rightarrow \Gamma[\check{s}](t)$ almost surely. Moreover, $\Delta \check{s}(t)=0$ gives us $\Gamma[s^{(n)}](t)=s^{(n+1)}\rightarrow \check{s}(t)$ almost surely by the above, so $\check{s}(t)=\Gamma[\check{s}](t)$ almost surely for a dense set of times. As both processes are right-continuous they are thus indistinguishable, so $\check{s}$ is a fixed point of $\Gamma$.

	Using the law of the iterated logarithm, we can find $\Omega_0\in\mathcal{F}_\infty$ with $\mathbb{P}(\Omega_0)=1$ so that, for all $\omega \in \Omega_0$, $\inf_{r\in[0,\delta]} \theta W_r(\omega) \geq - \delta^{1/3}$ infinitely often as $\delta \downarrow 0$. By the above, we may assume $\Omega_0$ is such that $s^{(n)}(t)(\omega) \rightarrow s(t)(\omega)$ as $n\rightarrow \infty$ for all $\omega \in \Omega_0$ and every $t$ in the set $\mathbb{T}(\omega):=\{t>0 : \Delta s(t)(\omega)=0\}$. As the latter is co-countable, for each $\omega \in \Omega_0$, we can find a strictly decreasing sequence of $\mathcal{F}_\infty$-random times $\delta_k=\delta_k(\omega)\downarrow 0$ in $\mathbb{T}(\omega)$ with $\inf_{r\in[0,\delta_{k}]}\theta W_r \geq - \delta_{k}^{1/3}$ on $\Omega_0$.  Fix $n,k\geq1 $ and consider an arbitrary $y \in [0,s^{(n)}(\delta_k)-s_0]$. Since $ s^{(n)}$ is increasing and continuous with $s^{(n)}(0)=s_0$, we can find $\delta \in [0, \delta_k]$ so that $s^{(n)}(\delta) - s_0 =  y$.
	Now define $\tilde{B}_t:=\inf_{r\in[0,t]} \sqrt{2\kappa - \theta^2} B_r$. If $ \tilde{B}_{\delta}  \geq - \delta^{1/3}$ and $x +  \tilde{B}_{\delta}  \leq s_0 + y + \delta^{1/3}_{k} $, then $x \leq s_0 + y +2 \delta^{1/3}_{k} $, so
	\begin{align}\label{eq:bounding_nu0_below}
		\nu_0([s_0, s_0+y+2\delta^{1/3}_{k}]) \geq \int_{s_0}^\infty &\mathbb{P}\bigl( x +  \tilde{B}_{\delta}   \leq s_0 + y + \delta_{k}^{1/3}  \mid \mathcal{F}_\infty \bigr)\nu_0(\mathrm{d}x)\nonumber\\
		& - \mathbb{P}(\tilde{B}_{\delta}  \leq - \delta^{1/3} \mid  \mathcal{F}_\infty)\nu_0([s_0,\infty)).
	\end{align}
	Since $\delta \leq \delta_k$, we have $\inf_{r\in[0,\delta]}\theta  W_r  \geq - \delta_{k}^{1/3}$ on $\Omega_0$. Thus, using twice that  $s_0 + y = s^{(n)}(\delta)$, and using also that $s^{(n)}$ is increasing, we get
	\begin{align*}
		\frac{1}{\lambda \kappa}\int_{s_0}^\infty \mathbb{P}\bigl( x +  \tilde{B}_{\delta}   \leq    s_0 + y + \delta_{k}^{1/3}  \mid \mathcal{F}_\infty \bigr)\nu_0(\mathrm{d}x) & \geq \Gamma[s^{(n)}](\delta) - s_0\\
		& \geq s^{(n)}(\delta)- s_0 = y
	\end{align*}
	on $\Omega_0$. By independence, $\mathbb{P}(\tilde{B}_{\delta}  \leq - \delta^{1/3} \mid  \mathcal{F}_\infty) \leq C \delta^{2/3} \leq C\delta_{k}^{2/3}$ for a constant $C>0$, so \eqref{eq:bounding_nu0_below} now gives
	\begin{equation}\label{eq:PJC_bound_initial}
		\frac{1}{\lambda \kappa} \nu_0([s_0, s_0+y+2\delta^{1/3}_{k}]) \geq y - C\delta_k^{2/3}
	\end{equation}
	for all $y \in [0,s^{(n)}(\delta_k)-s_0]$ on $\Omega_0$.
	Since $\delta_k \in \mathbb{T}$, we can send $n\rightarrow \infty$ to see that \eqref{eq:PJC_bound_initial} holds for all $y \in [0,s(\delta_k)-s_0]$ on $\Omega_0$. Sending $k\rightarrow \infty$, we conclude that, almost surely,
	\[
	\frac{1}{\lambda \kappa}	\nu_0([s_0, s_0+y ]) \geq y \;\; \text{for all } y \leq  s(0)-s_0.
	\]
	Since $s(0)\geq s_0$, we must have $s(0)=s_0$, as we would otherwise contradict \eqref{eq:initial_aligned}.
\end{proof}

\begin{rem}We note that the last part of the above proof is similar in spirit to the weak convergence based arguments in \cite{hambly_aldair_etal}. Likewise, it shares some similarities with the earlier results of \cite{DIRT_SPA, LS} concerning the `physicality' of the limit points of related particle systems. See also Remark \ref{rem:physical_minimal} concerning the notion of `physical' solutions.
\end{rem}

To continue our analysis, we introduce the finite $\mathcal{F}_t$-adapted random measures
\begin{equation}\label{eq:measure_nu}
	\nu(t,\mathrm{d}x)=\int_{s_0}^\infty \mathbb{P}( X_t^y \in \mathrm{d}x ,\,t<\tau^y \mid \mathcal{F}_\infty) \nu_0(\mathrm{d}y),
\end{equation}
for $t\geq 0$, along with their left-limits
\[
\;\;  \nu(t-,\mathrm{d}x)=\int_{s_0}^\infty \mathbb{P}( X_t^y \in \mathrm{d}x ,\,t\leq \tau^y \mid \mathcal{F}_\infty) \nu_0(\mathrm{d}y),
\]
for $t> 0$. We can take a version of $(\nu(t,\mathrm{d}x))_{t\geq 0}$ in \eqref{eq:measure_nu} such that $t\mapsto \int\phi(x)\nu(t,\mathrm{d}x)$ is c\`adl\`ag for all $\phi\in C_b(\mathbb{R})$ and this is the version we work with throughout. We also note that, by Lemma \ref{prop:no_blow_up_timing} below, $\nu(t,\mathrm{d}x)$ has a density in $L^1 \cap L^\infty$ for all $t>0$, so we have that $t\mapsto \nu(t,[a,b])$ is c\`adl\`ag for any $a,b \in \mathbb{R}$.

Straightforward considerations reveal that each $\nu(t-,\mathrm{d}x)$ is supported on $[s(t-),\infty)$, for $t>0$, and that each $\nu(t,\mathrm{d}x)$ is simply the restriction of $\nu(t-,\mathrm{d}x)$ to $(s(t),\infty)$, for $t>0$. Likewise, $\nu(0,\mathrm{d}x)=\nu_0$ when $s(0)=s_0$, and otherwise $\nu(0,\mathrm{d}x)$ agrees with the restriction of $\nu_0$ to $(s(0),\infty)$. We shall make use of this below.

\begin{lem}[Restarted solutions]\label{lem:restart}
	Let $\tau$ be any $(\mathcal{F}_t)_{t\geq0}$-stopping time, and let $\hat{\mathfrak{D}}^\uparrow_\mathbb{H}$ denote the space $\mathfrak{D}^\uparrow_\mathbb{H}$ defined instead for the new filtration $\hat{\mathcal{F}}_r:=\mathcal{F}_{\tau+r}$, $r\geq0$. Write $\hat{s}_0:=s(\tau)$ and $\hat{\nu}_0(\mathrm{d}x):=\nu(\tau,\mathrm{d}x)$ on the event $\{\tau < \infty\}$. Then we can express any solution $s\in \mathfrak{D}^\uparrow_\mathbb{H}$ to \eqref{eq:MV} as
	\begin{equation}\label{eq:paste_soln}
		s(t)= s(t)\mathbf{1}_{\{t<\tau\}} + \hat{s}(t-\tau)\mathbf{1}_{\{t\geq\tau\}},\quad t\geq0,
	\end{equation}
	where $\hat{s}\in \hat{\mathfrak{D}}^\uparrow_\mathbb{H}$ solves \eqref{eq:MV}
	on $\{\tau < \infty\}$ with $\hat{s}(0)=\hat{s}_0$ for initial conditions $\hat{s}_0$ and $\hat{\nu}_0$.
\end{lem}
\begin{proof}
	We can assume $\tau$ is finite, as the arguments are analogous when restricting to $\{\tau < \infty\}$. Define $\hat{W}^t_r := W_{r+t} - W_t $ and $\hat{B}^t_r:=B_{r+t} - B_t $ for $t,r\geq0$.  Fix $h>0$ and set $\sigma:= \sqrt{2\kappa - \theta^2}$. Noting that $\tau \perp (B,W,s) \,|\, \mathcal{F}^\infty$, we get
	\begin{align*}
		s(\tau+h) &= \mathbb{P}(\tau^y \leq t + h \mid \mathcal{F}_{\infty})|_{t=\tau}  = \mathbb{P}(\tau^y\leq t  \mid \mathcal{F}_{\infty})|_{t=\tau} +   \mathbb{P}(t< \tau^y\leq t+h  \mid \mathcal{F}_{\infty})   \\
		&=s(\tau) + \!\int_{s_0}^\infty \!\!\mathbb{P}\bigl( \exists r\leq h :  X^y_t +\sigma \hat{B}^t_r + \theta \hat{W}^t_r  \leq s(t +r) ,\, t <\tau^y \mid \mathcal{F}_{\infty}\bigr)\bigr|_{t=\tau} \nu_0(\mathrm{d}y)\\
		& = s(\tau) + \!\int_{s(\tau)}^\infty \mathbb{P}\bigl( \exists r\leq h : z + \sigma \hat{B}^t_r + \theta \hat{W}^t_r \leq  s(t +r) \mid \mathcal{F}_{\infty}\bigr) \bigr|_{t=\tau}\,\nu(\tau,\mathrm{d}z)\\
		& = s(\tau) + \!\int_{s(\tau)}^\infty \mathbb{P}\bigl( \exists r\leq h : z + \sigma \hat{B}^\tau_r + \theta \hat{W}^\tau_r \leq  s(\tau +r) \mid \mathcal{F}_{\infty}\bigr)\nu(\tau,\mathrm{d}z),
	\end{align*}
	almost surely. Now define $\hat{\tau}^z=\tau^{z,\hat{s}}$ by \eqref{eq:Gamma_map} for the Brownian motions $\hat{B}:=\hat{B}^\tau$ and $\hat{W}:=\hat{W}^\tau$, and the barrier $\hat{s}(r):=s(\tau+r)$ for $r\geq0$. Then the above gives that
	\[
	s(\tau+h) = s(\tau) + \int_{s(\tau)}^\infty \mathbb{P}( \hat{\tau}^z \leq h \mid \mathcal{F}_{\infty})\nu(\tau,\mathrm{d}z)
	\]
	almost surely, for every $h>0$. By right-continuity, it follows that $\hat{s}$ is a solution to \eqref{eq:MV} for the given filtration and the given initial conditions.
\end{proof}

Armed with this lemma, we can prove the following property of the jump discontinuities of the freezing front.

\begin{prop}[Lower bound on jump discontinuities]\label{prop:min_jump} Let $s\in  \mathfrak{D}^{\uparrow}_{\mathbb{H}} $ be a solution to \eqref{eq:MV}. Then
	\begin{equation}\label{eq:minimal_initial_jump}
		s(0) -  s_0 \geq	\inf \Bigl\{ y >0 : \frac{1}{\lambda \kappa } \nu_{0}([s_0,s_0 + y]) < y \Bigr\} 
	\end{equation}
	almost surely, and, for any $\mathcal{F}_t$-stopping time $\tau$, we have
	\begin{equation}\label{eq:minimal_jump}
		s(\tau) -  s(\tau-)  \geq	\inf \Bigl\{ y >0 : \frac{1}{\lambda \kappa } \nu(\tau-,[s(\tau-),s(\tau-) + y]) < y \Bigr\} 
	\end{equation}
	on $\{\tau \in (0,\infty)\}$ almost surely.
\end{prop}
\begin{proof}
	Fix $\tau$ and suppose, for a contradiction, that there is a non-negligible event $\Omega_0$ on which
	\[
	\Delta s(\tau) <	\inf \Bigl\{ y >0 :  \frac{1}{\lambda \kappa }\nu(\tau-,[s(\tau-),s(\tau-) + y]) < y \Bigr\} .
	\]
	Writing $s(\tau)=s(\tau-)+\Delta s(\tau)$ and using that $\nu(\tau,\mathrm{d}x)$ equals the restriction of $\nu(\tau-,\mathrm{d}x)$ to $(s(\tau),\infty)$, the same calculation as for \eqref{eq:jump-tau_k} gives
	\begin{equation}\label{eq:jump_restart}
		\inf \Bigl\{ y >0 :  \frac{1}{\lambda \kappa }\hat{\nu}_0([\hat{s}_0,\hat{s}_0 + y]) < y \Bigr\} > 0
	\end{equation}
	on $\Omega_0$ with $ \hat{\nu}_0(\mathrm{d}x):= \nu(\tau,\mathrm{d}x)$ and $ \hat{s}_0:= s(\tau)$. Now apply Lemma \ref{lem:restart} to express $s$ in terms of the restarted solution $\hat{s}\in \hat{\mathfrak{D}}^\uparrow_\mathbb{H}$ with these initial conditions. With probability one, we can ensure that $\hat{W}_{h}\leq -\sqrt{h}$ infinitely often as $h\downarrow 0$. Intersecting $\Omega_0$ with this event, we thus have a non-negligible event on which we can proceed analogously to the arguments in the proof of \cite[Proposition 3.5]{LS}. In this way, it follows that \eqref{eq:jump_restart} leads to a contradiction of the right-continuity $\hat{s}(h)\downarrow \hat{s}_0$ as $h\downarrow 0$, with non-zero probability, so \eqref{eq:minimal_jump} holds. The proof of \eqref{eq:minimal_initial_jump} is identical without the step involving restarting of the solution.
\end{proof}

Analogously to \eqref{eq:initial_jump_constraint}, an application of the conditional dominated convergence theorem in \eqref{eq:MV} gives that, for all $\mathcal{F}_t$-stopping times $\tau$, we must have
\begin{equation}\label{eq:stop_time_jump}
	s(\tau )= s(\tau-) + \frac{1}{\lambda \kappa }\nu(\tau -, [s(\tau-),s(\tau)])
\end{equation}
almost surely. By continuity of the measures, we can observe that the right hand-sides of \eqref{eq:minimal_initial_jump} and \eqref{eq:minimal_jump} satisfy the constraints \eqref{eq:minimal_initial_jump} and \eqref{eq:stop_time_jump} respectively. One would therefore expect that the jumps of the minimal solution can be described in this way, unless of course those values can somehow fail to be attainable. Luckily, things turn out to align with intuition: the minimal solution does indeed attain the lower bound on the jump sizes.

\begin{thm}[Minimal jumps]\label{prop:min_soln_jumps} Let $s\in  \mathfrak{D}^{\uparrow}_{\mathbb{H}}$ be the minimal solution. Then \eqref{eq:minimal_initial_jump} and \eqref{eq:minimal_jump} hold with equality.
\end{thm}
\begin{proof}
	Suppose, for a contradiction, that there is an $\mathcal{F}_t$-stopping time $\tau$ and a non-negligible event $\Omega_0$ such that the inequality in \eqref{eq:minimal_jump} is strict on this event. Set
	\[
	\hat{s}_0:=s(\tau -)+ \inf \Bigl\{ y >0 :  \frac{1}{\lambda \kappa } \nu(\tau-,[s(\tau-),s(\tau-) + y]) < y \Bigr\} 
	\]
	and define $\hat{\nu}_0(\mathrm{d}x):=\nu(\tau-,\mathrm{d}x \cap (	\hat{s}_0,\infty))$. Then 
	\[
	\inf \Bigl\{ y >0 :  \frac{1}{\lambda \kappa } \hat{\nu}_0([	\hat{s}_0,	\hat{s}_0 + y]) < y \Bigr\} =0,
	\]
	so Theorem \ref{thm:min_soln} and Proposition 
	\ref{prop:initial_aligned} yields the existence of a minimal solution $\hat{s} \in \hat{\mathfrak{D}}^{\uparrow}_{\mathbb{H}}$ with $\hat{s}(\tau)=\hat{s}_0$, where $\hat{\mathfrak{D}}^{\uparrow}_{\mathbb{H}}$ is as in the proof of Lemma \ref{lem:restart}. Pasting this together with the original solution as in \eqref{eq:paste_soln}, we can argue as in the proof of Lemma \ref{lem:restart}, only in reverse, to obtain a solution $\tilde{s}\in \mathfrak{D}^{\uparrow}_{\mathbb{H}}$. By construction, $\tilde{s}(\tau)=\hat{s}_0 < s(\tau)$ on the non-negligible event $\Omega_0$, so $s \nleq \tilde{s}$, which contradicts the minimality of $s$.
\end{proof}

\begin{rem}[Physicality of the minimal solution]\label{rem:physical_minimal} Consider the setting $\theta =0$ with initial conditions $s_0=0$ and $\nu_0$, where $\nu_0$ is some probability measure. Since $\theta =0$, the freezing front $s$ is now deterministic and so are the measures $\nu(t,\mathrm{d}x)$. Taking also $\kappa:=1/2$ and $\alpha := 1/2\lambda $, the problem \eqref{eq:MV} then becomes equivalent to the one studied in \cite{cuchiero, DNS, hambly_ledger_sojmark_2018}. Following \cite{DIRT_SPA} (see also \cite{cuchiero, DNS, hambly_ledger_sojmark_2018}), we can declare that an increasing right-continuous solution $t\mapsto s(t)$ is \emph{physical} if
	\[
	\Delta s(t) =  \inf \{ y>0 :  \alpha \,\nu(t-,[s(t-),s(t-)+y]) < y  \},\;\; \text{for all}\;\; t\geq0,
	\]
	with $s(0-):=s_0$ and $\nu(0-,\mathrm{d}x):=\nu_0$. In \cite[Proposition 2.3]{cuchiero} it was shown, by different but related arguments, that there is a minimal solution to \eqref{eq:MV} in the above setting, and \cite[Theorem 6.5]{cuchiero} then confirmed that this minimal solution is \emph{physical} in the above sense. The former relies on the introduction of an appropriate topology which Theorem \ref{thm:min_soln} avoids (for a topological statement, see Remark \ref{rem:M1}). The latter involves suitable finite approximating particle systems with perturbed initial conditions, while Proposition \ref{prop:min_soln_jumps} with $\tau=t$ provides a more direct approach based on restarting the system in agreement with given initial conditions. We note also that \cite[Theorem 6.5]{cuchiero} assumes $\int_0^\infty x \,\nu_0(\mathrm{d}x)<\infty$ which can be dispensed with.
\end{rem}

\begin{prop}[Pathwise jump size constraints] For any solution $s\in \mathfrak{D}^{\uparrow}_{\mathbb{H}}$ we have
	\begin{equation}
		s(t) -  s(t-) \geq  \inf \Bigl\{ y >0 :  \frac{1}{\lambda \kappa } \nu(t-,[s(t-),s(t-) + y])< y \Bigr\},\quad \text{for all}\quad t>0, \label{eq:pathwise_min_jump}
	\end{equation}
	with probability 1. Moreover, if $s\in \mathfrak{D}^{\uparrow}_{\mathbb{H}}$ is the minimal solution, then there is equality in \eqref{eq:pathwise_min_jump} for all $t>0$ with probability 1.
\end{prop}
\begin{proof}
	By Lemma \ref{prop:no_blow_up_timing} below, $\nu(t,\mathrm{d}x)$ has a density for all $t>0$. Thus, the proof of the second claim follows in the same way as Theorem \ref{thm:traject_jump}, using Theorem \ref{prop:min_soln_jumps} in place of Theorem \ref{thm:minimality}. To establish \eqref{eq:pathwise_min_jump} we can follow the same arguments but with a slightly different ending (that we could also rely on for the second claim). Invoking Proposition \ref{prop:min_jump}, we can find $\Omega_\star\in \mathcal{F}_\infty$ with $\mathbb{P}(\Omega_\star)=1$ so that $s(t)(\omega) \geq \lim_{\varepsilon \downarrow 0} s(t;\varepsilon)(\omega)$ for all $t>0$ with $s(t)(\omega)\neq s(t-)(\omega)$, and we can then deduce $	\nu(\tau_k,[s(\tau_k),s(\tau_k)+y])\geq \lambda \kappa y$ for all $y\in [0,1/k]$ on $\{\tau_k < \infty\}$ by retracing the proof of Theorem \ref{thm:traject_jump} with the same definition of $\tau_k$. Relying now on the arguments in the proof of Proposition \ref{prop:min_jump} starting from \eqref{eq:jump_restart}, $\mathbb{P}(\tau_k < \infty)>0$ would lead to a contradiction of the right-continuity of the freezing front. Hence $\mathbb{P}(\tau_k < \infty)=0$, and so the claim follows as in the proof of Theorem \ref{thm:traject_jump}.
\end{proof}

\begin{rem}[Functional convergence to the minimal solution]\label{rem:M1} Let $s$ be the minimal solution. Since $0\leq s$ and $\Gamma[s]=s$, repeated applications of $\Gamma$ gives $s^{(n)}\leq s$ for all $n\geq 1$. In the proof of Proposition \ref{prop:initial_aligned}, we thus get $\check{s} \leq s$ and so $\check{s}=s$. Thus, $s^{(n)}$ converges almost surely to the minimal solution in the sense of pointwise convergence at all continuity points $t>0$ of the limit. If the convergence also holds at $t=0$, this is equivalent to almost sure convergence for Skorokhod's M1 topology on $\mathbf{D}^{\uparrow}_{\mathbb{H}}$, see \cite[Corollary~12.5.1]{whitt}. When \eqref{eq:initial_aligned} is satisfied, Proposition \ref{prop:initial_aligned} gives $s^{(n)}(0)=s(0)=s_0$, so there is M1 convergence. If the left-hand side, $\mathrm{LHS}$, of \eqref{eq:initial_aligned} is non-zero, we can redefine $s^{(n)}(0):=s_0 + \mathrm{LHS}$: then we have M1 convergence of $s^{(n)}$ to $s$, since $s(0)=s_0 + \mathrm{LHS}$ by Proposition \ref{prop:min_soln_jumps}. Alternatively, we can modify the M1 topology to not enforce convergence at the initial time (e.g., akin to \cite[Definition 4.5]{jakubowski}) and thus directly have M1 convergence of $s^{(n)}$ to $s$.
\end{rem}

\section{Proofs of Theorem \ref{thm:prob_rep_cont} and Theorem \ref{thm:existence_jump}}

\subsection{Proof of Theorem \ref{thm:prob_rep_cont}} 

\begin{proof}[Proof of Theorem \ref{thm:prob_rep_cont}]
	We consider only the case of a global continuous solution, i.e., $\tau = \infty$. The case $\tau <\infty$ follows analogously by restricting all arguments to $t \land \tau$. Let $(v,s)\in \mathfrak{C}(v_0,s_0)$ be an arbitrary global continuous weak solution.
	
	Given $t\mapsto s(t)$ and the associated filtration $(\mathcal{F}_t)_{t\geq 0}$, we can define $(X^y,\tau^y)$ as in \eqref{eq:prob_rep} for a Brownian motion $B$ that is independent of $(\mathcal{F}_t)_{t\geq 0}$. Crucially, we do \emph{not} insist that the pair $(X^y,\tau^y)$ satisfies \eqref{eq:prob_rep_0}: at this point, $s(t)$ is given exogenously rather than being defined by the second line of \eqref{eq:prob_rep_0}. Next, we define $\nu(t,\mathrm{d}x)$ as in \eqref{eq:measure_nu}, but in terms of $(X^y,\tau^y)$ and the initial condition $\nu_0(\mathrm{d}x):=-v_0(x)\mathrm{d}x$. We will then show that each $\nu(t,\mathrm{d}x)$ has a density that agrees with $v(t,x)$, thus showing that $v$ satisfies the first line of  \eqref{eq:prob_rep_0}. Since, $s(t)-s_0 = \frac{1}{\lambda\kappa }(\int_{s(t)}^\infty v(t,x)\mathrm{d}x - \int_{s_0}^\infty v_0(x)\mathrm{d}x )$, we can then also deduce the second line of \eqref{eq:prob_rep_0} from the definition of $\nu$, and so we will have the desired result.

	For each $t\geq0$, consider the random measure $\mu_t$ on the positive half-line obtained from the translation $\mu_t(A):=\nu(t,A+s(t))$. Let $G_{\varepsilon}(x,y)$ denote the Dirichlet heat kernel on the positive half-line, corresponding to the sub-probability transition density for a Brownian motion killed at the origin. Since $G_{\varepsilon}(x,0)=0$, applying It\^o's formula to $G_\varepsilon(x,X^z_{t\land \tau^z}-s(t\land \tau^z))$ and using the properties of $(\mathcal{F}_t)_{t\geq 0}$, straightforward manipulations yield
	\begin{align}\label{eq:1st_weak_cont}
		\langle \mu_t , G_\varepsilon(x, \cdot )\rangle  &=  	\int_{s_0}^\infty \mathbb{E}[	G_\varepsilon(x,X^z_{t\land \tau^z} - s(t\land \tau^z)) \mid \mathcal{F}_t ] v_0(z)\mathrm{d}z \nonumber \\
		& = \langle \mu_0 , G_\varepsilon(x, \cdot )\rangle  + \kappa \int_0^t \langle \mu_r , \partial_{yy} G_\varepsilon(x,\cdot) \rangle \mathrm{d}r \nonumber \\
		&+ \theta \int_0^t \langle \mu_r , \partial_{y} G_\varepsilon(x,\cdot) \rangle \mathrm{d}W_r -  \int_0^t \langle \mu_r , \partial_{y} G_\varepsilon(x,\cdot) \rangle \mathrm{d}s(r). 
	\end{align}
	Now consider the given solution $(v,s)$. For any $\phi \in C_b^2(\mathbb{R})$, we can use a smooth approximation of  $\phi(t,y):=\phi(t,y-s(t))$ in \eqref{eq:weak_cont} and pass to the limit to obtain
	\begin{equation*}
		\begin{aligned}
			\int_{s(t)}^{\infty} v(t,y) &\phi(y-s(t)) \mathrm{d}y  -  \int_{s_0}^{\infty}\!\!\!\!{v_0(y)} \phi(y-s_0) \mathrm{d}y = \kappa\!\int_0^t\! \int_{s(r)}^{\infty }\! v(r,y) \partial_{yy}\phi(y-s(r)) \mathrm{d}y\mathrm{d}r \\
			&-\int_0^t \int_{s(r)}^{\infty }v(r,y)\partial_y \phi(y-s(r)) \mathrm{d}y \mathrm{d}s(r)+ \lambda\kappa \phi(0) \bigl( s(t) - s_0\bigr)\\ &\qquad\qquad  +\theta \!\int_0^t\! \int_{s(r)}^{\infty } \!v(r,y) \partial_y\phi(y-s(r)) \mathrm{d}y\mathrm{d} W_r .
		\end{aligned}
	\end{equation*}
	Define the measure flow $\tilde{\mu}_t : \mathcal{B}(0,\infty) \rightarrow [0,1]$ by $\tilde{\mu}_t(a,b):= -\int_{a+s(t)}^{b+s(t)} v(t,z) \mathrm{d}z$, for $t\geq0$. Taking $\phi= G_\varepsilon (x,\cdot)$ in the above, for $x\geq 0$, we get the same expression as in \eqref{eq:1st_weak_cont} only for $\tilde{\mu}$ in place of $\mu$. Define $u^\varepsilon_t(x) := \langle \mu_t - \tilde\mu_t , G_\varepsilon(x,\cdot)\rangle $ and note that $u_0^\varepsilon\equiv 0$. Regarding the terms in \eqref{eq:1st_weak_cont}, the identity
	$\partial_y G_\varepsilon(x,y) = -\partial_x G_\varepsilon(x,y) - 2 p_\varepsilon'(x+y)$ gives the decomposition
	\[
	\langle \mu_r - \tilde\mu_r,\, \partial_y G_\varepsilon(x,\cdot) \rangle = -	\partial_x u_r^\varepsilon(x)  - \partial_x \mathrm{e}_r^\varepsilon(x),
	\]
	where we have introduced the error term
	\[
	\mathrm{e}_t^{\varepsilon}(x) := \langle \mu_t - \tilde\mu_t , 2p_\varepsilon(x+\cdot )\rangle,  \quad  p_{\varepsilon}(x)=\frac{1}{\sqrt{2\pi \varepsilon}}e^{-x^2/2\varepsilon}.
	\]
	Using this and $\partial_{yy}G_\varepsilon = \partial_{xx}G_\varepsilon$, taking the difference of the equation \eqref{eq:1st_weak_cont} for $\mu$ and $\tilde{\mu}$ yields
	\begin{align*}
		u_t^\varepsilon(x) &= \kappa \int_0^t \partial_{xx}u_r^\varepsilon(x) \mathrm{d}r  - \theta \int_0^t \partial_{x}u_r^\varepsilon(x) \mathrm{d}W_r  +   \int_0^t \partial_{x}u_r^\varepsilon(x) \mathrm{d}s(r) \\
		& \qquad - \theta \int_0^t \partial_{x}\mathrm{e}^\varepsilon_r(x)\mathrm{d}W_r +  \int_0^t \partial_{x}\mathrm{e}^\varepsilon_r(x)\mathrm{d}s(r) .
	\end{align*}
	
	Set $U^\varepsilon_t(x):= - \int_{x}^\infty u_t^\varepsilon(y) \mathrm{d}y$ so that $\partial_xU_t^\varepsilon (x)=u_t^\varepsilon(x)$. Note that $G_\varepsilon(0,\cdot)=0$ gives $u_t^\varepsilon(0)=0$, and the smoothness and exponential decay of $G_\varepsilon$ is inherited by $U^\varepsilon_t$, $u_t^\varepsilon$ and $\partial_x u_t^\varepsilon$. Thus, the boundary terms vanish when we integrate the equation for $u^\varepsilon_t$, giving
	\[
	\mathrm{d}U^\varepsilon_t(x) = \kappa \partial_x u^\varepsilon_t(x) \mathrm{d}t - \theta u_t^\varepsilon(x) \mathrm{d}W_t + u_t^\varepsilon(x) \mathrm{d}s(t) - \theta \mathrm{e}^\varepsilon_t(x)\mathrm{d}W_t+ \mathrm{e}^\varepsilon_t(x)\mathrm{d}s(t).
	\]
	Applying It\^o's formula (using that $t\mapsto s(t)$ is continuous and non-decreasing), we get
	\begin{align*}
		&\mathrm{d} \bigl(  e^{-s(t)} U^\varepsilon_t(x)^2 \bigr) = - e^{-s(t)}U^\varepsilon_t(x)^2 \mathrm{d}s(t) + \theta^2e^{-s(t)}\bigl( u_t^\varepsilon(x) + \mathrm{e}^\varepsilon_t(x)\bigr)^2 \mathrm{d}t \\
		&\qquad + 2\kappa e^{-s(t)} U^\varepsilon_t(x) \partial_x u_t^\varepsilon(x) \mathrm{d}t + 2 e^{-s(t)} U^\varepsilon_t(x) u_t^\varepsilon(x) \mathrm{d}s(t)  + 2 e^{-s(t)} U^\varepsilon_t(x) \mathrm{e}_t^\varepsilon(x) \mathrm{d}s(t)  \\
		&\qquad - 2\theta e^{-s(t)} U^\varepsilon_t(x) u_t^\varepsilon(x) \mathrm{d}W_t - 2\theta e^{-s(t)} U^\varepsilon_t(x) \mathrm{e}_t^\varepsilon(x) \mathrm{d}W_t  
	\end{align*}
	Integrating in space and performing a few estimates (using integration by parts, Cauchy--Schwarz, and Young's inequality), we arrive at
	\begin{align*}
		& \mathbb{E} \bigl[e^{-s(t)}\Vert U_t^\varepsilon \Vert_2^2\bigr
		] \leq  
		- \mathbb{E} \Bigl[ \int_0^t e^{-s(r)} \Vert U_r^\varepsilon\Vert_2^2 	\mathrm{d}s(r)  \Bigr] +  \delta \mathbb{E} \Bigl[  \int_0^t e^{-s(r)}\Vert U_r^{\varepsilon} \Vert_2^2 \, \mathrm{d} s(r) \Bigr]    \\
		& \quad - 2\kappa\! \int_0^t \mathbb{E} \bigl[e^{-s(r)} \Vert u_r^\varepsilon \Vert_2^2 \bigr] \mathrm{d}r  +  \theta^2 \! \int_0^t 
		\mathbb{E} \bigl[ e^{-s(r)}\Vert u_r^{\varepsilon}\Vert_2^2 \bigr]  \, \mathrm{d}r+\delta  \! \int_0^t 
		\mathbb{E} \bigl[ e^{-s(r)}\Vert u_r^{\varepsilon}\Vert_2^2 \bigr]  \, \mathrm{d}r \\
		&\quad + C_\delta  \mathbb{E} \Bigl[  \int_0^t e^{-s(r)}\Vert \mathrm{e}_r^{\varepsilon} \Vert_2^2 \, \mathrm{d} s(r) \Bigr] +  C_\delta  \!\int_0^t \mathbb{E} \bigl[e^{-s(r)}\Vert \mathrm{e}_r^{\varepsilon}\Vert_2^2\bigr] \mathrm{d} r,
	\end{align*}
	where $\delta>0$ can be chosen as small as we like. Here we have also used that the stochastic integrals are true martingales and that
	\[
	\int_0^\infty \!2 U^\varepsilon_t(x) u_t^\varepsilon(x)\mathrm{d}x = \int_0^\infty \partial_x \bigl( U_t^\varepsilon(x) \bigr)^2 \mathrm{d}x =  -
	U_t^\varepsilon(0)^2  \leq 0.
	\]
	Since $|\theta|< \sqrt{2\kappa}$, we can take $\delta>0$ sufficiently small so that
	\[
	\mathbb{E} \bigl[e^{-s(t)}\Vert U_t^\varepsilon \Vert_2^2\bigr
	] \leq C_\delta  \mathbb{E} \Bigl[  \int_0^t e^{-s(r)}\Vert \mathrm{e}_r^{\varepsilon} \Vert_2^2 \, \mathrm{d} s(r) \Bigr] +  C_\delta  \!\int_0^t \mathbb{E} \bigl[e^{-s(r)}\Vert \mathrm{e}_r^{\varepsilon}\Vert_2^2\bigr] \mathrm{d} r.
	\]
	Finally, from the definition, we have
	\[
	\mathrm{e}_r^{\varepsilon} \leq C e^{-2x^2/\varepsilon} \varepsilon^{-1/2}  \bigl( (\mu_r - \tilde\mu_r)(0,z) + e^{-z^2/2\varepsilon} \bigr),
	\]
	for any $z>0$, and hence
	\[
	\Vert \mathrm{e}_r^{\varepsilon} \Vert_2^2 \leq C \varepsilon^{-1/2} \bigl( \mu_r(0,z)^2 + \tilde\mu_r(0,z)^2 + e^{-z^2/\varepsilon} \bigr),
	\]
	for any $z>0$.  Since $\mu_t(0,z), \tilde\mu_t(0,z) \leq z \Vert v_0 \Vert_{L^\infty}$, we can take $z=\varepsilon^p$ with $p\in(1/4,1/2)$ to get $\Vert \mathrm{e}_r^{\varepsilon} \Vert_2^2 \rightarrow 0$ as $\varepsilon \rightarrow 0$ with a bound that is independent of $r>0$. Consequently, 
	\[
	\lim_{\varepsilon \downarrow 0}  \mathbb{E} \bigl[e^{-s(t)}\Vert U_t^\varepsilon \Vert_2^2\bigr
	]  = 0,
	\]
	for all $t\geq 0$. From this, we readily deduce that, for each $t\geq 0$, we have
	$\langle \mu_t , \phi \rangle = \langle \tilde\mu_t , \phi \rangle $
	for all $\phi \in C_b(\mathbb{R})$ almost surely. Since  $(\mu_t)_{t\geq0}$ and $(\tilde\mu_t)_{t\geq0}$  are c\`adl\`ag processes, it follows that  they are indistinguishable, which completes the proof.
\end{proof}

\subsection{Proof of Theorem \ref{thm:existence_jump}}

\begin{proof}[Proof of Theorem \ref{thm:existence_jump}] 
	Since we assume that \eqref{eq:ic_constraint} holds, we can take a solution $s\in \mathfrak{D}^{\uparrow}_{\mathbb{H}}$ to \eqref{eq:MV} with $s(0)=s_0$, by Theorem \ref{thm:min_soln} and Proposition \ref{prop:initial_aligned}. Now consider the corresponding $\mathcal{F}_t$-adapted random measures $\nu(t,\mathrm{d}x)$ defined in \eqref{eq:measure_nu}, for each $t\geq0$. By Lemma \ref{prop:no_blow_up_timing} below, we get that each $\nu(t,\mathrm{d}x)$ has a (random) density $v(t,\cdot)\in L^1\cap L^\infty(\mathbb{R})$ with respect to Lebesgue which is supported on $(s(t),\infty)$ and satisfies $\Vert v(t,\cdot) \Vert_{L^\infty} \leq \Vert v_0 \Vert_{L^\infty}$ for all $t>0$. Moreover, the fact that $s(0)=s_0$ along with the definition of $\nu(0,\mathrm{d}x)$ also gives $v(0,\cdot)=v_0$ almost everywhere, so $(v,s) \in \mathfrak{D}(v_0,s_0)$. We will show that the pair $(v,s)$ satisfies the weak formulation \eqref{eq:weak_jump} of Definition \ref{def:weak_jumps}. Fix an arbitrary $\phi \in C_b^\infty(\mathbb{R})$. Using It\^o's formula and performing some straightforward manipulations, based in particular on properties of the filtration $(\mathcal{F}_t)_{t\geq 0}$, we arrive at
	\begin{align}\label{eq:1st_weak}
		\int_{s(t)}^\infty &v(t,x) \phi(t,x)\mathrm{d}x  =  	\mathbb{E}[	\phi(t,X^y_{t\land \tau^y}) \mid \mathcal{F}_t ]  - 		\mathbb{E}[\phi(\tau^y,X^y_{ \tau^y}) \mathbf{1}_{\{t\geq \tau^y\}}\mid \mathcal{F}_t ] \nonumber \\
		& = \int_{s_0}^{\infty}\!\!\!\! v_0(x) \phi(0,x) \mathrm{d}x +  \int_0^t \! \int_{s(r)}^{\infty } \!v(r,x)\partial_r \phi(r,x) \mathrm{d}x \mathrm{d}r \nonumber  \\ &\qquad+ \kappa\!\int_0^t \int_{s(r)}^{\infty } v(r,x) \partial_{xx}\phi(r,x) \mathrm{d}x\mathrm{d}r  \nonumber \\
		&\qquad+ \theta\!\int_0^t\! \int_{s(r)}^{\infty } \!v(r,x) \partial_x \phi(r,x) \mathrm{d}x\mathrm{d}{ W_{r}} - \mathbb{E}[\phi(\tau^y,X^y_{ \tau^y}) \mathbf{1}_{\{t\geq \tau^y\}}\mid \mathcal{F}_t ].
	\end{align}
	Now, $s\mapsto s(t)$ is c\`adl\`ag and hence any realisation has countably many jumps in $(0,t]$. Moreover, we can observe from \eqref{eq:MV} that, when $\Delta s(r) \neq 0$, we have $r=\tau^y$ if and only if  $r\leq \tau^y$ and $X^y_{r} \in (s(r-), s(r-) + \Delta s(r)]$. Therefore, using properties of $\mathcal{F}_{t}$ and the definition of $v$, we can write
	\begin{align}\label{eq:2nd_weak}
		\int_{s_0}^\infty\mathbb{E}&[\phi(\tau^y,X^y_{ \tau^y}) \mathbf{1}_{\{t\geq \tau^y\}} \mathbf{1}_{\{\Delta s(\tau^y) \neq 0\}}\mid \mathcal{F}_t ]v_0(y)\mathrm{d}y  \nonumber\\
		&= 	\sum_{0 < r \leq t} \int_{s_0}^\infty\mathbb{E}[\phi(r,X^y_{r})\mathbf{1}_{\{r \leq \tau^y\}}\mathbf{1}_{\{X^y_{r-} \in (s(r-), s(r-) + \Delta s(r)] \}}\mid \mathcal{F}_t ]v_0(y)\mathrm{d}y \nonumber\\
		&= \sum_{0 < r \leq t} \int_{s(r-)}^{s(r-)+\Delta s(r)}v(r-,x)\phi(r,x)\mathrm{d}x. 
	\end{align}
	Next, we can observe that, by the expression for $s(t)$ in \eqref{eq:MV} and the properties of $\mathcal{F}_t$,
	\begin{align}\label{eq:3rd_weak}
		\int_{s_0}^\infty &\mathbb{E}[\phi(\tau^y,X^y_{ \tau^y}) \mathbf{1}_{\{t\geq \tau^y\}} \mathbf{1}_{\{\Delta s(\tau^y)=0\}} \mid \mathcal{F}_t ] v_0(y)\mathrm{d} y  \nonumber \\ &= 	\int_{s_0}^\infty 	\!\mathbb{E}[\phi( \tau^y,s(\tau^y) ) \mathbf{1}_{\{t\geq \tau^y\}} \mathbf{1}_{\{\Delta s(\tau^y) = 0\}} \mid  \mathcal{F}_t ] v_0(y)\mathrm{d} y\nonumber  \\ &= \lambda \kappa \int_0^t \phi(r,s(r)) \mathbf{1}_{\{\Delta s(r)=0\}}\,\mathrm{d} s(r) \nonumber\\
		&= \lambda \kappa \int_0^t \phi(r,s(r-)) \mathrm{d} s(r)  -  \sum_{0<r\leq t} \phi(r,s(r-))\lambda \kappa\Delta s(r).
	\end{align}
	Finally, it remains to observe that, when $\Delta s(r)\neq 0$, the dynamics of \eqref{eq:MV} enforce
	\begin{align}\label{eq:4th_weak}
		\lambda \kappa \Delta s(r) &= \int_{s_0}^\infty \mathbb{P}\bigl(X_{r} \in (s(r-) , s(r)] ,\, r\leq \tau^y \mid \mathcal{F}_t \bigr) v_0(y)\mathrm{d}y \nonumber \\ 
		&= \int_{s(r-)}^{s(r)} v(r-,x)\mathrm{d}x.
	\end{align}
	Combining \eqref{eq:1st_weak}, \eqref{eq:2nd_weak}, \eqref{eq:3rd_weak}, and \eqref{eq:4th_weak}, we finally obtain the weak formulation \eqref{eq:weak_jump}. This completes the proof.
\end{proof}

\section{Proofs of Theorem \ref{thm:jump_nonzero_prob} and Theorem \ref{prop:initial_global_cont}} 
\label{Sect_BU}

In this section, we will prove Theorems \ref{thm:jump_nonzero_prob} and \ref{prop:initial_global_cont}, using the conditional McKean--Vlasov problem \eqref{eq:MV} and, respectively, Theorems \ref{thm:prob_rep_cont} and \ref{thm:existence_jump}. Since the statements are concerned with deterministic initial conditions, we restrict to this throughout. Our overall approach is inspired by \cite[Theorem 2.1]{LS}, but that result is tailored to Dirac initial conditions and does not generalise to our setting, so significant new ideas are needed. Nevertheless, we wish to align our presentation as closely as possible with that of \cite{LS} and related works. Thus, we introduce the new parameters $\sigma := \sqrt{2\kappa - \theta^2}$ and $\alpha := \nu_0((s_0,\infty)) / \lambda \kappa $, and note that we can then rewrite \eqref{eq:MV} as
\begin{equation}\label{eq:McKean}
	\left\{
	\begin{aligned}
		X_t &= \bar{X}_0 + \sigma B_t + \theta W_{t} - \alpha L_t \\[1pt]
		\tau &= \inf\{ t \geq 0 : X_t  \leq 0 \} \\[1pt]
		L_t &= \mathbb{P}(\tau \leq t \mid \mathcal{F}_t )
	\end{aligned} \right.
\end{equation}
with $s(t)=s_0 + \alpha L_t$, where $\bar{X}_0$ is a random variable that is independent of $(B,W)$ and distributed on $(0,\infty)$ according to the (rescaled and translated) probability measure $\bar{\nu}_0(\mathrm{d}x):= \nu_0(\mathrm{d}x+s_0)/\nu_0((s_0,\infty)) $. We single out $L_t$, in place of the front $s(t)$, to align with \cite{LS} and work directly with the conditional cdf of the hitting time $\tau$.

\subsection{Initial considerations}
\label{Sect_BU_Common}

By analogy with the measures $\nu(t,\mathrm{d}x)$ in \eqref{eq:measure_nu}, we define
\begin{equation}\label{eq:nu_t}
	\bar{\nu}(t,\mathrm{d}x):=\mathbb{P}(X_t \in \mathrm{d}x,\,t<\tau \mid \mathcal{F}_t),
\end{equation}
for all $t\geq0$, where $X$ and $\tau$ are given by \eqref{eq:McKean}. We begin by confirming that, for $t>0$, $\bar{\nu}_t$ always has a density $V_t$ with respect to the Lebesgue measure on $(0,\infty)$ which already has a few important implications. When $L_0=0$ (i.e., $s(0)=s_0$), we have $	\bar{\nu}(0,\mathrm{d}x)=\bar{\nu}_0$.

\begin{lem}[Existence of densities]\label{prop:no_blow_up_timing} For all $t>0$, the sub-probability measure $\bar{\nu}(t,\mathrm{d}x)$ in \eqref{eq:nu_t} has a density $V_t\in L^1\cap L^\infty (\mathbb{R})$ with $\Vert V_t \Vert_{L^\infty} \leq 1/\sqrt{2\pi\sigma^2 t}$. Moreover, $\Vert V_{t+s} \Vert_{L^\infty} < \Vert V_t \Vert_{L^\infty} $ for all $s,t>0$; and, if $\bar{\nu}(0,\mathrm{d}x)$ has a density $V_0\in L^\infty$, then $\Vert V_t \Vert_{L^\infty}\leq \Vert V_0 \Vert_{L^\infty}$ for all $t\geq 0$. In particular, no discontinuities can occur after time $t=\alpha^2/2\pi\sigma^2$, and no discontinuities can occur after a given time $t$ on the event $\Vert V_t \Vert_{L^\infty} < \alpha^{-1}$.
\end{lem}
\begin{proof}
	Similarly to the proof of \cite[Lemma 2.2]{LS}, we can observe that, for any $s>0$ and $t\geq0$, we have
	\begin{equation}\label{eq:density_proof}
		\bar{\nu}(t+s,A)\leq \int_A \int_0^\infty \tilde{p}_{s}(y - x - \theta \widetilde{W}_s + \alpha \widetilde{L}_s) \bar{\nu}(t,\mathrm{d}x)\mathrm{d}y, 
	\end{equation}
	for all $A\in\mathcal{B}(0,\infty)$, where $\widetilde{W}_s:=W_{t+s}-W_t$, $\widetilde{L}_s:=L_{t+s}-L_t$, and $ \tilde{p}_{s}$ denotes the Normal density of $\sigma B_s$. The latter is bounded by $1/\sqrt{2\pi\sigma^2 s}$ and $\bar{\nu}(t,\mathrm{d}x)$ is a sub-probability measure, so the first conclusion follows. Now, exploiting $\bar{\nu}(t,\mathrm{d}x)=V_t(x)\mathrm{d}x$, we deduce the second claim from \eqref{eq:density_proof} using that $\tilde{p}_{s}$ integrates to one on $\mathbb{R}$. Finally, \eqref{eq:stop_time_jump} gives that we must have $s(t)=s(t-)$ when $\Vert V_t \Vert_{L^\infty} < \alpha^{-1}$ which completes the proof.
\end{proof}

Throughout this section, we shall need the following notation, which we present here in the form of a definition for easy reference.

\begin{defn}[$\mathcal{A}$, $\mathcal{U}$, and $\mathcal{V}$] For any times $s\leq t$ and constants $m,k\in\mathbb{R}$ as well as $l,\delta>0$, we define
	\begin{align}
		\mathcal{A}_{t; m, l} &:= \{ f: [0, t] \to \mathbb{R} \;\;\textrm{s.t.}\;\; |f(r) - mr| < l \textrm{ for all } r \in [0, t] \}, \label{eq:mathcal_A}\\[4pt]
		\mathcal{U}_{s,t;\delta} &:=\{ f:[s,t] \to \mathbb{R} \;\;\textrm{s.t.}\;\; |f(r) - f(s)| \leq \delta \textrm{ for all } r \in [s, t]  \},\label{eq:mathcal_U}\\[4pt]
		\mathcal{V}_{s,t; k} &:=\{ f: [s, t] \to \mathbb{R} \;\;\textrm{s.t.}\;\; f(r) < k \textrm{ for all } r \in [s, t] \},\label{eq:mathcal_V}
	\end{align}
	where we will, e.g., write $f\in \mathcal{A}_{t; m, l} \cap \mathcal{U}_{s,u; \delta}$ if the restriction of $f$ to $[0,t]$ is in $\mathcal{A}_{t; m, l}$ and the restriction of $f$ to $[s,u]$ is in $\mathcal{U}_{s,u; \delta}$.
\end{defn}

\subsection{Proof of Theorem \ref{thm:jump_nonzero_prob}}\label{Sect:proof_thm_jump_nonzero_prob}

For any given initial condition $(v_0,s_0)$, as in the statement of Theorem \ref{thm:jump_nonzero_prob}, we let $\nu_0(\mathrm{d}x)=-v_0(x)\mathrm{d}x$. Then, also $\bar{\nu}_0(\mathrm{d}x)= \nu_0(\mathrm{d}x+s_0)/\nu_0((s_0,\infty))$ has a density, which we denote by $V_0$, and the condition of Theorem \ref{thm:jump_nonzero_prob} amounts to the following fact: $V_0(c)>\alpha^{-1}$ for some $c \in (0,\infty)$ with $V_0$ right-continuous at $c$.

We set $\delta := 2(V_0(c)-\alpha^{-1})/3>0$. By the right-continuity at $c$, we can find $h>0$ such that $|V_0(c)-V_0(x)|\leq \delta/2$ for $x\in c+(0,h)$, and hence
\begin{equation}\label{eq:initial_density_bounds}
	\alpha^{-1} + \delta \leq V_0(x) \leq \alpha^{-1} + 2\delta, \qquad \textrm{for all } x \in c + (0,h).
\end{equation}
The proof of Theorem \ref{thm:jump_nonzero_prob} exploits \eqref{eq:initial_density_bounds} to show that the dynamics of the problem must then force a blow-up in finite time with strictly positive probability.

The key to this lies in the two lemmas that we shall establish next. They yield suitable localised control over the mass and moment of the sub-probability measures $\bar{\nu}(t,\mathrm{d}x)$ from \eqref{eq:nu_t}, as we move along particular trajectories of the driving Brownian motion $W$. We note that the overall idea of deriving the finite time blow-up from a mass versus moment comparison is similar to \cite{LS}. However, the main result on blow-up in that work is for a Dirac initial condition, so there is no need to consider the finer localised control that we require here. In turn, we must rely on completely new ideas for the two lemmas that follow and their subsequent use in finalising the proof of Theorem \ref{thm:jump_nonzero_prob}.

Before embarking on the detailed arguments, we briefly outline our overall strategy for the proof. Setting $I=(0,h)$, the idea is to first show that, with $L_t$ in \eqref{eq:McKean} evolving continuously, trajectories of $\theta W_t$ that have a suitable negative trend can transport the mass in \eqref{eq:initial_density_bounds} to a corresponding right-neighbourhood $I$ of the freezing front, while ensuring a limited loss of mass locally and keeping the localised moment small, by controlling the diffusive spreading relative to the fast transport. This is made precise in Lemma~\ref{Lem:Stage1}. Next, as long as $L_t$ remains continuous, the localised mass and moment can be seen to satisfy a simple quadratic constraint: roughly, an amount of mass $M$ on $I$ necessitates a moment of order at least $M^2$ when restricted to $I$, where the proportionality is dictated by $\alpha$. This arises from the way in which each increment of mass absorbed from $I$ contributes to shifting the profile towards the front, and it is made precise by Lemma~\ref{Lem:Stage2}. Combining these observations, we are finally able to conclude that, for suitable trajectories of $\theta W_t$, the condition \eqref{eq:initial_density_bounds} will force the mass near the freezing front to become larger than what is allowed by the aforementioned quadratic constraint, and so continuity of $L$ cannot be maintained.


\begin{lem}[Local mass and moment control]
	\label{Lem:Stage1} Fix $\varepsilon>0$, $c>0$, and $I=(0,h)$ with $h>0$. Given $l>0$, there exists $t_0 > 0$, and $m < 0$ such that
	\[
	\bar{\nu}({\bar{\tau}},I) \geq \bar{\nu}_0(c + I) - \varepsilon
	\qquad
	\textrm{and}
	\qquad
	\int_I x \bar{\nu}({\bar{\tau}},\mathrm{d}x) \leq \int_{c + I} (x-c) \bar{\nu}_0(\mathrm{d}x) + \varepsilon,
	\]
	on the event $(\theta W_t)_{t\in[0,t_0]}\in\mathcal{A}_{t_0,m,l}$,
	for some $W$-measurable time $\bar{\tau} \in(0,t_0]$, provided $L$ is continuous on $[0,t_0]$ on the aforementioned event.
\end{lem}

\begin{proof}
	First, we define the auxiliary process
	\[
	H_s  := c + \theta W_s  - \alpha L_s, \qquad \text{for } s\geq 0,
	\]
	and let
	\[
	\bar{\tau}:= \inf\{s>0: H_s   \leq 0 \},
	\]
	where $\bar{\tau}$ is $W$-measurable, since $H$ is defined in terms of $L$ and $W$. Since $H_0=c>0$ and paths $s\mapsto H_s$ are right-continuous, we have $\bar{\tau}>0$. Given $m<0$, let $t_0=t_0(m,l)$ be the first time for which the deterministic linear path $s\mapsto c+ l + ms $ hits zero. On the event $(\theta W_s)_{s\in[0,t_0]}\in \mathcal{A}_{t_0;m,l}$, we have $\theta W_s \leq ms + l$ for all $s\leq t_0$, so we deduce that
	\[
	0<  {\bar{\tau}} \leq t_0 = - \frac{c+l}{m}
	\]
	on this event, since $\alpha L$ is non-negative. For the remainder of the proof, we restrict attention to the event $(\theta W_s)_{s\in[0,t_0]}\in \mathcal{A}_{t_0;m,l}$ and assume $L$ is continuous on $[0,t_0]$ on this event, as in the statement of the lemma. Now rewrite $X$ as
	\[
	X_s = (X_0-c) + \sigma B_s + H_s, \qquad \text{for } s\geq 0,
	\]
	and notice that $H_{\bar{\tau}}=0$, since ${\bar{\tau}} \leq t_0$ and we are assuming continuity of $L$ on $[0,t_0]$. Therefore, we can write
	\begin{align}
		\bar{\nu}({{\bar{\tau}}},I) &= \mathbb{P}\bigl (X_0 - c + \sigma B_{\bar{\tau}} \in I, \; \inf_{s\leq {\bar{\tau}} }\{ X_0 - c + \sigma B_s + H_s  \}>0 \mid \mathcal{F}_\infty \bigr) \nonumber\\
		&=\mathbb{P}(X_0 - c + \sigma B_{\bar{\tau}} \in I \mid \mathcal{F}_\infty)\nonumber \\
		&\qquad - \mathbb{P}\bigl (X_0 - c + \sigma B_{\bar{\tau}} \in I, \; \inf_{s\leq {\bar{\tau}}}\{ X_0 - c + \sigma B_s + H_s  \}\leq 0 \mid \mathcal{F}_\infty \bigr)\nonumber\\
		&:=C_1({\bar{\tau}}) - C_2({\bar{\tau}}).\label{eq:Defn_C_1_and_C_2}
	\end{align}
	For the first probability, observe that
	\begin{align}
		C_1({\bar{\tau}}) &\geq  \mathbb{P}\bigl(X_0-c + \sigma B_{\bar{\tau}} \in I \mid |B_{\bar{\tau}}| \leq r/\sigma,\; W\bigr) \mathbb{P}(|B_{\bar{\tau}}| \leq r/\sigma \mid \mathcal{F}_\infty) \nonumber \\
		&\geq \mathbb{P}\bigl(X_0\in (c+r , c+h -r)\bigr) \bigl(1-\mathbb{P}(|B_{\bar{\tau}}| > r/\sigma \mid \mathcal{F}_\infty ) \bigr),
		\label{eq:C_1_bound}
	\end{align}
	for any $r\in(0,h/2)$. By first fixing $r>0$ sufficiently small, we can ensure
	\begin{align*}
		\mathbb{P}\bigl(X_0 \in (c+r , c+h -r)\bigr) & = \bar{\nu}_0(c+I) - \bar{\nu}_0((c,c+r]) -\bar{\nu}_0([c+h-r,c+h)) \\
		&\geq \bar{\nu}_0(c+I) - \varepsilon/8.
	\end{align*}
	Next, given this $r>0$, we can use the fact that ${\bar{\tau}} \leq t_0$ to force
	\begin{align}
		\mathbb{P}(|B_{\bar{\tau}}| > r/\sigma \mid \mathcal{F}_\infty ) &= \mathbb{P}(\sqrt{{\bar{\tau}}}|B_1| > r/\sigma \mid \mathcal{F}_\infty ) \nonumber \\ &\leq \mathbb{P}(\sqrt{t_0}|B_1| > r/\sigma \mid \mathcal{F}_\infty) < \varepsilon/8,\label{eq:C_1_bound_2}
	\end{align}
	by simply taking $m<0$ sufficiently negative so that $t_0=t_0(m,l)$ is small. Inserting this in \eqref{eq:C_1_bound}, we finally obtain
	\begin{equation}\label{eq:final_bound_C_1}
		C_1({\bar{\tau}}) \geq \bigl( \bar{\nu}_0(c+I) - \varepsilon/8  \bigr)(1- \varepsilon/8) \geq  \bar{\nu}_0(c+I) - \varepsilon/4,
	\end{equation}
	for our choice of $r>0$ and $m<0$. Coming back to the second probability, note that $H_s \geq 0$ for all $s\leq {\bar{\tau}}$, so we have
	\[
	C_2({\bar{\tau}}) \leq \mathbb{P}\bigl (X_0 - c + \sigma B_{\bar{\tau}} \in I, \;  A_{\bar{\tau}} \leq 0 \mid \mathcal{F}_\infty \bigr),
	\]
	where
	\[
	A_{\bar{\tau}} := \inf_{s\leq {\bar{\tau}}}\{ X_0 - c + \sigma B_s\}.
	\]
	Splitting our analysis on the event $ |B_{\bar{\tau}}| \leq r/\sigma$ and its complement, we then get
	\begin{align*}
		C_2({\bar{\tau}})
		&\leq \mathbb{P}\bigl (X_0 - c + \sigma B_{\bar{\tau}} \in I, \; A_{\bar{\tau}}\leq 0, \; |B_{\bar{\tau}}| \leq r/\sigma \mid \mathcal{F}_\infty \bigr) +\mathbb{P}(|B_{\bar{\tau}}| > r/\sigma \mid \mathcal{F}_\infty) \\
		&\leq \mathbb{P}\bigl (X_0 - c + \sigma B_{\bar{\tau}} \in I, \; A_{\bar{\tau}} \leq 0, \; |B_{\bar{\tau}}| \leq r/\sigma \mid \mathcal{F}_\infty \bigr) + \frac{\varepsilon}{4},
	\end{align*}
	where we have used the bound from \eqref{eq:C_1_bound_2}.
	By lowering $r>0$ (and then taking $m<0$ more negative), if necessary, we can assume
	$\bar{\nu}_0(c+[-r,r])<\varepsilon/4$. Consequently, we can write
	\begin{align*}
		C_2({\bar{\tau}}) &\leq \mathbb{P}\bigl(X_0 \in (c-r, c+h+r), \; A_{\bar{\tau}} \leq 0 \mid \mathcal{F}_\infty \bigr) + \frac{\varepsilon}{4} \\ &\leq \mathbb{P}\bigl(X_0 \geq c+r, \; A_{\bar{\tau}} \leq 0 \mid \mathcal{F}_\infty \bigr) + \frac{\varepsilon}{2} \\
		&\leq \mathbb{P}\bigl(\, \inf_{s\leq {\bar{\tau}}}\{ r + \sigma B_s\} \leq 0 \mid \mathcal{F}_\infty \bigr) + \frac{\varepsilon}{2}
	\end{align*}
	Finally, it remains to observe that
	\[
	\mathbb{P}\bigl( \, \inf_{s\leq {\bar{\tau}}}\{ r + \sigma B_s\} \leq 0 \mid \mathcal{F}_\infty \bigr) \leq \mathbb{P}\bigl( \, \inf_{s\leq t_0}\{ r + \sigma B_s\} \leq 0 \bigr),
	\]
	since ${\bar{\tau}} \leq t_0$, where the right-hand side can be made less than $\varepsilon/4$ for $t_0>0$ small enough, independently of $W$. Therefore, we get
	\[
	C_2({\bar{\tau}}) \leq \mathbb{P}\bigl( \, \inf_{s\leq t_0}\{ r + \sigma B_s\} \leq 0 \bigr)+\frac{\varepsilon}{2} \leq  \frac{3}{4}\varepsilon,
	\]
	by taking $m<0$ sufficiently negative (so that $t_0$ is small enough). Together with \eqref{eq:Defn_C_1_and_C_2} and \eqref{eq:final_bound_C_1}, the above gives
	\[
	\bar{\nu}({{\bar{\tau}}},I) = C_1({\bar{\tau}}) - C_2({\bar{\tau}}) \geq \bar{\nu}_0(c+I) - \frac{\varepsilon}{4} - \frac{3\varepsilon}{4} = \bar{\nu}_0(c+I)  - \varepsilon,
	\]
	for the $W$-random time ${\bar{\tau}}$, where we recall that this inequality is strictly for the event $(\theta W_s)_{s\in[0,t_0]}\in \mathcal{A}_{t_0;m,l}$ and relies on the assumption that $L$ is continuous on $[0,t_0]$ on this event.
	
	Having proved the first claim of the lemma, we now proceed to verify the upper bound. This task is substantially simpler, since we can ignore the absorption at the origin. In particular, we simply have
	\[
	\int_I x \bar{\nu}(t,\mathrm{d} x) = \mathbb{E}[\mathbf{1}_{\{X_t \in I\}} \mathbf{1}_{\{t < \tau\}} X_t  \mid \mathcal{F}_\infty]   \leq  \mathbb{E}[\mathbf{1}_{\{X_t \in I\}} X_t  \mid \mathcal{F}_\infty],	
	\]
	noting that $I \subseteq (0,\infty)$. As before, we restrict to the event $(\theta W_s)_{s\in[0,t_0]}\in \mathcal{A}_{t_0;m,l}$ and assume continuity of $L$ on $[0,t_0]$ on this event. Then $H_{\bar{\tau}}=0$, since ${\bar{\tau}} \leq t_0$, and hence we get
	\[
	\int_I x \bar{\nu}({{\bar{\tau}}},\mathrm{d}x) \leq  \mathbb{E}[\mathbf{1}_{\{X_0 + \sigma B_{\bar{\tau}}   \in c+I\}} (X_0-c+\sigma B_{\bar{\tau}} )  \mid \mathcal{F}_\infty] ,
	\]
	Taking $m<0$ sufficiently negative, we can make $t_0=t_0(m,l)$ as close to zero as we like. Using again that ${\bar{\tau}}\leq t_0$, the previous inequality therefore implies
	\[
	\int_I x \bar{\nu}({{\bar{\tau}}},\mathrm{d}x) \leq  \mathbb{E}[\mathbf{1}_{\{X_0  \in c+I\}} (X_0-c )  \mid \mathcal{F}_\infty]  + \varepsilon = \int_{c+I} (x-c) \bar{\nu}_{0}(\mathrm{d}x) + \varepsilon,
	\]
	for $m<0$ sufficiently negative, as required. This completes the proof.
\end{proof}

Now that we can control the mass and the moment, we need a result that allows us to compare the two.


\begin{lem}[Local mass-moment inequality]
	\label{Lem:Stage2}
	Let $\varepsilon>0$ be given. Fix any $A \in \mathcal{B}(0,\infty)$, $l \in (0, \varepsilon/4]$, $m < 0$, and $t_0 > 0$, and set $k:=mt_0+l$. Then, there exists a large enough time $t_1>t_0$ with the property that: if $L$ is continuous on $[0,t_1]$, on the event
	\[
	E_{t_0,t_1}:=\{(\theta W_s)_{s\in[0,t_1]} \in \mathcal{A}_{t_0;m,l}\cap \mathcal{V}_{t_0,t_1;k}\},
	\]
	then we have
	\[
	\frac{\alpha}{2} \bar{\nu}(t,A)^2 \leq \int_A x \bar{\nu}(t,\mathrm{d}x) + \varepsilon,
	\]
	for all $t\in[0,t_0]$, on the aforementioned event $E_{t_0,t_1}$.
\end{lem}

\begin{proof}
	Let $A\in\mathcal{B}(0,\infty)$. To simplify the presentation, we  introduce the operations
	\[
	\mathbb{P}_t^A(\,\cdot \mid \mathcal{F}_\infty):=\mathbb{P}(\, \cdot\, , X_t \in A,\; t<\tau \mid \mathcal{F}_\infty )
	\]
	and, correspondingly,
	\[
	\mathbb{E}_t^A[\,\cdot \mid \mathcal{F}_\infty]:=\mathbb{E}[\, (\cdot)\,  \mathbf{1}_{\{X_t \in A\}} \mathbf{1}_{\{t<\tau\}} \mid \mathcal{F}_\infty ]
	\]
	for $t\geq0$. In other words, we are restricting the background space to the event $\{X_t \in A,\; t<\tau\}$. Given a time $t\geq0$, we wish to consider the absorbed process $X$ on $[t,\infty)$ as starting at time $t$ on the event $\{t<\tau\}$, and we can then study the process depending on the position of $X_t$. In particular, we can write the loss on $[t,s]$, for $s\geq t$, as
	\begin{align}\label{eq:lossA}
		L_{s } - L_t = \mathbb{P}( t < \tau \leq s  \mid \mathcal{F}_\infty)   = L^{A,t}_{ s } + L^{A^{\complement},t}_{s},
	\end{align}
	where
	\[
	L_{s}^{A,t}:= \mathbb{P}_t^A(\tau \leq s \mid \mathcal{F}_\infty).
	\]
	As in the statement of the lemma, we assume that, for some $t_1>t_0$ to be determined, $L$ is continuous on $[0,t_1]$ on the event $(\theta W_s)_{s\in[0,t_1]} \in \mathcal{A}_{t_0;m,l}\cap \mathcal{V}_{t_0,t_1;k}$. Given $t \in [0,t_0]$, we then have
	\begin{align}\label{eq:X_started_at_t}
		0 \leq X_{s \wedge \tau} = X_t + \sigma(B_{s \wedge \tau} - B_{t}) + \theta (W_{s \wedge \tau} - W_t) - \alpha (L_{s\wedge\tau}-L_t),
	\end{align}
	for all $s\in[t,t_1]$ on this event.
	Noting that $s\wedge \tau \geq t$ if $s\in[t,t_1]$ and $t<\tau$, optional stopping gives
	\[
	\mathbb{E}_t^A\bigl[B_{s\land \tau} - B_t \mid \mathcal{F}_\infty] = 	\mathbb{E}[\mathbb{E}[B_{s\land \tau} - B_t \mid X_0, (B_r)_{r\leq t} , W] \mathbf{1}_{\{t<\tau\}} \mathbf{1}_{\{X_t \in A\}}\mid \mathcal{F}_\infty\bigr]= 0.
	\]
	Therefore, rearranging and applying $\mathbb{E}_t^A$ conditional on $\mathcal{F}_\infty$ in \eqref{eq:X_started_at_t} yields
	\begin{equation}\label{eq:LA_bound}
		\alpha \mathbb{E}^A_t [ L_{s \wedge \tau} - L_t \mid \mathcal{F}_\infty] \leq  \int_A x\bar{\nu}(t,\mathrm{d}x) + \mathbb{E}^A_t[ \theta W_{s \wedge \tau} - \theta W_t \mid \mathcal{F}_\infty],
	\end{equation}
	for $s\in[t,t_1]$, on the event $(\theta W_s)_{s\in[0,t_1]} \in \mathcal{A}_{t_0;m,l}\cap \mathcal{V}_{t_0,t_1;k}$, where we have also used that
	\[
	\mathbb{E}_t^{A}[X_t \mid \mathcal{F}_\infty]=\int_A{x}\bar{\nu}(t,\mathrm{d}x).
	\]
	Next, in view of $m < 0$, $k = mt_0 + l$, and $0<l\leq \varepsilon/4$,  we have
	\[
	\mathbb{E}^A_t[  \theta W_{s \wedge \tau} -  \theta W_t \mid \mathcal{F}_\infty] \leq 2 l \leq \varepsilon/2,
	\]
	for all $s\in[t,t_1]$ with $t\in[0,t_0]$, whenever $(\theta W_s)_{s\in[0,t_1]} \in \mathcal{A}_{t_0;m,l}\cap \mathcal{V}_{t_0,t_1;k}$, since  $s \land \tau \geq t$ if $t<\tau$. Furthermore, by definition of $L^{A,t}$ and \eqref{eq:lossA}, we have
	\begin{align*}
		\mathbb{E}^A_t [ L_{s \wedge \tau} - L_t \mid \mathcal{F}_\infty] = \int_0^\infty (L_{s \wedge r} - L_t)\mathrm{d}L_{r}^{A,t} \geq \int_t^\infty L^{A,t}_{s \wedge r} \mathrm{d}L_{r}^{A,t}.
	\end{align*}
	Observe that $s\mapsto L_{s}^{A,t}$ is increasing, and $\Delta L^{A,t}_s\leq \Delta L_s$, for all $s\geq t$, so $L^{A,t}$ is continuous whenever $L$ is. In particular, $L^{A,t}$ is continuous and of finite variation, and so we can explicitly evaluate the integral
	\[
	\int_t^\infty L^{A,t}_{s \wedge r} \mathrm{d}L_{r}^{A,t} = \frac{1}{2} ( L^{A,t}_{s} )^2 + L^{A,t}_{s}(L^{A,t}_{\infty} - L^{A,t}_{s} ),
	\]
	for $s\in[t,t_1]$, assuming $L$ is continuous on $[0,t_1]$. Putting these observations together, it follows from \eqref{eq:LA_bound} that
	\begin{equation}\label{eq:LA_final_bound}
		\frac{\alpha}{2} ( L^{A,t}_{s} )^2  \leq  \alpha \int_t^\infty L^{A,t}_{s \wedge r} \mathrm{d}L_{r}^{A,t} \leq  \int_A x\bar{\nu}(t,\mathrm{d}x) + \frac{\varepsilon}{2},
	\end{equation}
	for $s\in[t,t_1]$, on the event $(\theta W_s)_{s\in[0,t_1]} \in \mathcal{A}_{t_0;m,l}\cap \mathcal{V}_{t_0,t_1;k}$. Since Brownian motion hits every level with probability one, we have 
	\[
	\mathbb{P}_t^A(s < \tau \mid \mathcal{F}_\infty)  \leq \mathbb{P}(s < \tau \mid \mathcal{F}_\infty) \rightarrow 0,
	\]
	almost surely, as $s \to \infty$.
	By construction of the event $(\theta W_s)_{s\in[0,t_1]} \in \mathcal{A}_{t_0;m,l}\cap \mathcal{V}_{t_0,t_1;k}$, on which we in particular have $\theta W_s \leq k$ for all $s\in[t_0,t_1]$, we can thus take the deterministic time $t_1>t_0$ large enough such that, on this event,
	\[
	\frac{\alpha}{2} (L_{t_1}^{A,t})^2=\frac{\alpha}{2} \bigl( \bar{\nu}(t,A) - \mathbb{P}_t^A( t_1 < \tau \mid \mathcal{F}_\infty) \bigr)^2 \geq \frac{\alpha}{2} \bar{\nu}(t,A)^2 - \frac{\varepsilon}{2},
	\]
	for all $t\in[0,t_0]$. Setting this into \eqref{eq:LA_final_bound} gives the result.
\end{proof}


From here, we have everything we need to complete the proof of Theorem \ref{thm:jump_nonzero_prob}, by exploiting \eqref{eq:initial_density_bounds} and putting together the two lemmas.

\begin{proof}[Proof of Theorem \ref{thm:jump_nonzero_prob}]
	Given the shape of the initial condition, we let $c>0$, $h>0$, and $\delta>0$ be fixed as in \eqref{eq:initial_density_bounds}. Next, for reasons that will be clear at the end of this proof, we then choose $\gamma=\gamma(\alpha, \delta) \in (0,1)$ such that
	\begin{equation}
		\label{eq:SetUpForContradiction}
		\alpha \delta (1 - \gamma)^2 > 6\gamma,
	\end{equation}
	and finally we fix a value $\varepsilon=\varepsilon(h,\gamma, \delta) > 0$ satisfying 
	\begin{equation}\label{eq:fix_varepsilon}
		\varepsilon \leq \gamma \delta h \min(1, h).
	\end{equation}
	
	Fixing also some $l\in(0,\varepsilon/4]$, we let $t_0>0$ and $m<0$ be given by Lemma \ref{Lem:Stage1}. For these values, we then let $t_1>t_0$ be given by Lemma \ref{Lem:Stage2} with $A=I$, where we recall our notation $I=(0,h)$. Now let $(v,s)\in \mathfrak{D}(v_0,s_0)$ be any c\`adl\`ag weak solution in the sense of Definition \ref{def:weak_jumps} and suppose, for a contradiction, that $\mathbb{P}(\varsigma = +\infty)=1$. Then $(v,s)$ is a continuous weak solution for which Theorem  \ref{thm:prob_rep_cont} applies and gives us a probabilistic representation which we rewrite as \eqref{eq:McKean}. By the assumption of continuity, the process $L$ from \eqref{eq:McKean} is in particular continuous on $[0,t_1]$, so Lemma \ref{Lem:Stage2} applies. Moreover, we let ${\bar{\tau}}$ be the $W$-measurable random time provided by Lemma \ref{Lem:Stage1}. Since $0<{\bar{\tau}} \leq t_0$, it follows from Lemmas \ref{Lem:Stage1} and \ref{Lem:Stage2} that
	\begin{equation}\label{eq:bound_nu_alpha}
		\frac{\alpha}{2}(\bar{\nu}_0(c + I) - \varepsilon)^2 
		\leq \frac{\alpha}{2} \bar{\nu}({{\bar{\tau}}},I)^2
		\leq \int_I x \bar{\nu}({{\bar{\tau}}},\mathrm{d}x) + \varepsilon
		\leq \int_{c + I} (x-c) \bar{\nu}_0(\mathrm{d}x) + 2\varepsilon
	\end{equation}
	on the event $\{(\theta W_s)_{s\in[0,t_1]} \in \mathcal{A}_{t_0;m,l}\cap \mathcal{V}_{t_0,t_1;k}\}$. Furthermore, by \eqref{eq:initial_density_bounds}, we have
	\[
	\bar{\nu}_0(c + I) \geq h(\alpha^{-1} + \delta)
	\]
	and
	\[
	\int_{c + I}(x-c) \bar{\nu}_0(\mathrm{d}x) 
	\leq  (\alpha^{-1} + 2\delta) \int_0^h x \mathrm{d}x 
	= \frac{1}{2}(\alpha^{-1} + 2\delta) h^2.
	\]
	Together with \eqref{eq:bound_nu_alpha}, this yields
	\[
	\alpha ((\alpha^{-1} + \delta)h - \varepsilon)^2 
	\leq  (\alpha^{-1} + 2\delta) h^2 + 4\varepsilon.
	\]	
	On the left-hand side, we can use $\varepsilon \leq \gamma \delta h < (\alpha^{-1} + \delta)h$, and we can use $\varepsilon \leq \gamma \delta h^2$ on the right-hand side. Exploiting this, and dividing through by $h^2$, we in turn get
	\[
	\alpha (\alpha^{-1} + (1- \gamma)\delta)^2 
	\leq \alpha^{-1} + (2+4\gamma)\delta,
	\]
	which simplifies to
	\[
	\alpha\delta(1 - \gamma)^2 \leq 6\gamma.
	\]
	This contradicts our choice of $\gamma$ in \eqref{eq:SetUpForContradiction}, so we conclude that $\mathbb{P}(\varsigma < + \infty)>0$.
\end{proof}

\subsection{Proof of Theorem \ref{prop:initial_global_cont}}

We shall first establish an auxiliary lemma that, in particular, rules out an immediate blow-up, and then we exhibit an event of strictly positive probability on which the solution remains continuous for all time, thus proving Theorem \ref{prop:initial_global_cont}.

Our argument proceeds in three steps. First, Lemma~\ref{lem:density_control} confirms that the smoothing effect of the diffusion causes the density $V_t$ from Lemma \ref{prop:no_blow_up_timing} to drop strictly below the critical value $\alpha^{-1}$ near the interface, for a short random amount of time. Secondly, we introduce an auxiliary McKean--Vlasov problem \eqref{eq:aux_system_shift_mass}, obtained by shifting the initial condition by a small amount towards the interface. We then use a comparison argument to show that, for trajectories of $\theta W_t$ with a controlled positive trend in some interval of time, the loss process $L_t$ from \eqref{eq:McKean} is dominated by that of the auxiliary problem. Since the auxiliary problem is constructed so that its density stays subcritical near the front, this extends the continuity of $L_t$ to a certain fixed deterministic time $t_\star>0$. Finally, beyond this time $t_\star$, we let $\theta W_t$ push the bulk of mass away from the front for some additional time, allowing the diffusive spreading to flatten the entire density by enough that no further blow-ups can occur.

\begin{lem}[Density control near the interface]\label{lem:density_control}
	Consider a solution $s\in \mathfrak{D}^{\uparrow}_{\mathbb{H}}$ to \eqref{eq:MV} with $s(0)=s_0$. Equivalently, $L\in \mathfrak{D}^{\uparrow}_{\mathbb{H}}$ with $L_0=0$ in \eqref{eq:McKean}.
	Suppose that $\bar{\nu}_0$ has a density $V_{0}$ with $V_0\leq C\mathbf{1}_{(0,x_0)} + D\mathbf{1}_{[x_0,\infty)}$, for some constants $x_0>0$ and $D\geq C >0$.
	Then there exists a $W$-measurable random time $\tau_{\star}>0$ and a fixed $x_{\star}>0$ such that $V_{t}(x)\leq(1-\delta(t))C$ on $(0,x_\star)$ for all $t\in(0,\tau_\star)$, where $\delta(t)$ is increasing in $t$ and $\delta(t)>0$ for $t>0$.
\end{lem}
\begin{proof}
	Let $p_{t}$ denote the density of the Normal random variable $\sigma B_{t}$. Set $x_{\star}:=x_{0}/4$. Let $H_t:=  \alpha L_{t} - \theta W_{t}$. Since $H_0=0$, by right continuity we can find a $W$-measurable time $\tau_{0}>0$ such that $H_t< x_\star$ for $t\leq\tau_0$. Ignoring the absorption at the origin, we have the bound
	\begin{align*}
		\nu_{t}([a,b]) & \leq\mathbb{P}\bigl(a-\sigma B_{t}+H_t \leq X_{0}\leq b-\sigma B_{t}+H_t \mid \mathcal{F}_\infty\bigr)\\ &=\mathbb{E}\Bigl[\int_{a-\sigma B_{t}+H_t}^{b-\sigma B_{t}+H_t}V_{0}(x)\mathrm{d}x \mid \mathcal{F}_\infty\Bigr] =\int_{-\infty}^{b+H_t}\int_{a}^{b}V_{0}(z+H_t-y)p_{t}(y)\mathrm{d}z\mathrm{d}y.
	\end{align*}
	Since $\sup_{s\leq r}H_s$ is of course increasing in $r$, we can find another random time $0<\tau_\star\leq\tau_0$ so that $c_{\star}:=x_\star- \sup_{s\leq \tau_\star}H_s >0$ and $De^{-c_{\star}^{2}/2\sigma^2 t}\leq C$ for $t\leq \tau_{\star}$.
	
	We proceed with computations similar to those in the proof of \cite[Proposition 2.1]{DNS} (for the analysis of \eqref{eq:McKean} with $\theta=0$), but we note that the observations made here are of a slightly different nature. Throughout, we restrict to $t\leq \tau_{\star}$ and consider an arbitrary interval $[a,b] \subseteq [0,x_{\star})$. We split the outer integral at $y=-b-|H_t|-c_{\star}$. Since $b<x_{\star}$, if $y\geq-b-|H_t|-c_{\star}$, then $b+H_t-y\leq 2b + 2(H_t \land 0) + c_{\star} <4x_{\star}=x_{0}$, so the first integral satisfies
	\begin{align*}
		\int_{-b-|H_t|-c_{\star}}^{b+H_t}\int_{a}^{b}V_{0}(z+H_t-y)p_{t}(y)\mathrm{d}z\mathrm{d}y & \leq C(b-a)\int_{-b-|H_t|-c_{\star}}^{b+H_t}p_{t}(y)\mathrm{d}y.
	\end{align*}
	For the other integral, we have
	\begin{align*}
		\int_{-\infty}^{-b-|H_t|-c_{\star}}\int_{a}^{b}V_{0}(z+H_t-y)\mathrm{d}z\,p_{t}(y)\mathrm{d}y & \leq D(b-a)\int_{-\infty}^{-b-|H_t|-c_{\star}}p_{t}(y)\mathrm{d}y.
	\end{align*}
	Changing variables $y\mapsto y+c_{\star}$, we can then note that
	\begin{align*}
		\int_{-\infty}^{-b-|H_t|-c_{\star}}p_{t}(y)\mathrm{d}y & \leq e^{-c_{\star}^{2}/2\sigma^2 t}\int_{-\infty}^{-b-|H_t|}p_t(y)\mathrm{d}y =e^{-c_{\star}^{2}/2\sigma^2 t}\int_{b+|H_t|}^{\infty}p_{t}(y)\mathrm{d}y.
	\end{align*}
	Recalling the control $De^{-c_{\star}^{2}/2\sigma^2 t}\leq C$ for $t\leq \tau_{\star}$, we thus conclude that
	\begin{align*}
		\nu_{t}([a,b]) & \leq C(b-a)\left(\int_{-b-|H_t|-c_{\star}}^{b+H_t}p_{t}(y)\mathrm{d}y+\int_{b+|H_t|}^{\infty}p_{t}(y)\mathrm{d}y\right)\\
		& \leq C(b-a)\left(1-\int_{-\infty}^{-c_{\star}-x_{\star}-|H_t|}p_{t}(y)\mathrm{d}y - \int_{b+H_t}^{b+|H_t|}p_{t}(y)\mathrm{d}y\right),
	\end{align*}
	for every $t\leq \tau_{\star}$. Since $[a,b]\subseteq[0,x_\star)$ was arbitrary, this finishes the proof.
\end{proof}

With the above lemma in place, we are now ready to prove Theorem \ref{prop:initial_global_cont}. Differently from Section \ref{Sect:proof_thm_jump_nonzero_prob}, here it will be possible to follow more closely the ideas from \cite{LS}, specifically the proof of \cite[Proposition 2.9]{LS}. Nevertheless, several changes are required and a more careful comparison argument must be employed.

\begin{proof}[Proof of Theorem \ref{prop:initial_global_cont}]
	Let the solution of minimal temperature increase from Theorem \ref{thm:minimality} be given. We will work with its probabilistic representation, which we rewrite in the form \eqref{eq:McKean} with $\nu_0(\mathrm{d}x)=-v_0(x)\mathrm{d}x$. By the assumption on $v_0$, it follows from Lemma \ref{lem:density_control} that the density $V_t$ of $\bar{\nu}(t,\mathrm{d}x)$, defined in \eqref{eq:nu_t}, lies strictly below $\alpha^{-1}$ on a neighbourhood of the origin on some random interval $(0,\tau_\star)$ with $\tau_\star>0$. Thus, the right-hand side of \eqref{eq:fragility_lower_bound} is zero for all $t$ in this interval, and hence Theorem \ref{thm:traject_jump} gives that, with probability 1, we cannot have a jump discontinuity on $[0,\tau_\star)$. Consequently, $\varsigma= \inf \{ t>0: L(t)\neq L(t-) \}$ is strictly positive with probability 1. This establishes the first claim. Towards the second claim, our assumptions on $v_0$ allow us to find an arbitrarily small $\varepsilon\in (0,x_0)$ such that $\bar{\nu}_0([0,\varepsilon])<\alpha^{-1}\varepsilon$ with $\bar{\nu}_0([0,\varepsilon+x])<\alpha^{-1}(\varepsilon+x)$ infinitely often as $x\downarrow \varepsilon$. Write $\bar{\nu}_0([0,\varepsilon])=(1-\delta)\alpha^{-1}\varepsilon$ for a small $\delta>0$. Given a continuous function $g:[0,\infty)\rightarrow\mathbb{R}$ with $g(0)=0$, we can consider the problem
	\begin{equation}
		\label{eq:aux_system_shift_mass}
		\begin{cases}
			\widehat{X}_t = Z -\delta \varepsilon + \sigma B_t + g(t)   -  \alpha \widehat{L}_t \\
			\widehat{\tau}= \inf \{t>0: \widehat{X}_t \leq0\}\\
			\widehat{L}_t=\bar{\nu}_0([0,\varepsilon]) +\mathbb{E}[ \mathbb{P}(\widehat{\tau}\leq t \mid Z>\varepsilon)],
		\end{cases}
	\end{equation}
	where $Z$ is distributed according to $\bar{\nu}_0$. As $x\mapsto \bar{\nu}_0([0,\varepsilon+x])$ satisfies \eqref{eq:initial_aligned}, this problem has a solution with  $\widehat{L}_0=\bar{\nu}_0([0,\varepsilon])$ and minimal jumps. We can use this as the basis for a refined version of the no-crossing lemma \cite[Lemma 2.8]{LS}. Suppose $f:[0,\infty)\rightarrow\mathbb{R}$ is a continuous function with $f(0)=0$ and $f(t)>g(t)-\delta\varepsilon$ for all $t\in[0,t_0)$, given some $t_0>0$, and assume $\widetilde{L}$ is continuous on $[0,t_0)$, where $\widetilde{L}$ solves
	\[
	\begin{cases}
		\widetilde{X}_t = Z + \sigma B_t + f(t)   -  \alpha \widetilde{L}_t \\
		\widetilde{\tau}= \inf \{t>0: \widetilde{X}_t \leq0\}\\
		\widetilde{L}_t= \mathbb{P}(\widetilde{\tau}\leq t )
	\end{cases}
	\]
	with $\widetilde{L}_0=0$ and minimal jumps (in the sense of Theorem \ref{prop:min_soln_jumps}). By construction, $\widehat{X}_0<\widetilde{X}_0$. If, moreover, $\widehat{X} \leq \widetilde{X}$ on $[0,t_0)$, then
	\begin{align*}
		\widetilde{L}_t =\widetilde{ L}_{t-} & \leq
		\bar{\nu}_0([0,\varepsilon]) +  \int_{\varepsilon}^{\infty}\mathbb{P}(\inf_{s < t} \widetilde{X}_s < 0 \mid Z=z)\bar{\nu}_{0}(\mathrm{d}z) \\
		&\leq \bar{\nu}_0([0,\varepsilon]) +  \int_{\varepsilon}^{\infty}\mathbb{P}(\inf_{s < t} \widehat{X}_s < 0 \mid Z=z)\bar{\nu}_{0}(\mathrm{d}z) = \widehat{L}_{t-} \leq \widehat{L}_{t} 
	\end{align*}
	for all $t<t_0$.
	At the same time, arguing as in \cite[Lemma 2.8]{LS}, we can let $s_0>0$ be the first time $\widehat{X}_{s} \geq \widetilde{X}_{s}$ and observe that we must then have 
	\[
	\alpha(\widehat{L}_{s_0} - \widetilde{L}_{s_0}) = g(s_0) -  f(s_0) - \delta \varepsilon < 0,
	\]
	if $s_0 < t_0$, which is a contradiction since the left-hand side is non-negative. Thus, we must have $\widehat{X} \leq \widetilde{X}$ on $[0,t_0)$, and in turn $\widetilde{L}\leq \widehat{L} $ on $[0,t_0)$.

	Take $x_0>0$ such that $V_0(x)\leq\alpha^{-1}$ on $[\varepsilon, \varepsilon + x_0)$, and let $x_\star:=x_0/4$. Let $t_\star>0$ be such that $\alpha  \widehat{L}_t\leq x_\star/3$ and $g(t)-\delta\varepsilon \geq - x_\star/3$ for all $t\leq t_\star$. This is possible by right continuity of $\widehat{L}$,  since $\widehat{L}_0$ can be made as close to zero as we like by decreasing $\varepsilon\in(0,x_0)$ if necessary. Lowering  $ t_\star>0$, if required, we can assume $ e^{-c^{2}/2\sigma^2 t_\star}\Vert V_0\Vert_{L^\infty}\leq \alpha^{-1}$, where $c:=x_\star - \sup_{s\leq t_\star}(\alpha \widehat{L}_s - g(s) + \delta\varepsilon)>0$. By jump minimality, it then follows from the proof of Lemma \ref{lem:density_control}, as in the first part of this proof, that  $\widehat{L}$ is continuous on $[0,t_\star)$. Take the above function $g$ of the form $g(s):=ms$, for some $m>0$ to be determined. Note that, as long as $\delta\varepsilon \leq x_\star /3$, we then have $g(s)-\delta \varepsilon\geq -x_\star/3$ for all $s\geq0$, so we restrict to small enough $\varepsilon>0$ for which this is the case. 
	
	From the first part of this proof, we have $\varsigma \geq \tau_\star$, where $\tau_\star >0$ is given by Lemma \ref{lem:density_control}. Recall the family of paths $\mathcal{A}$ defined in \eqref{eq:mathcal_A}, and take $l:=\delta \varepsilon>0$. On the event $(\theta W_t)_{t\leq t_\star} \in \mathcal{A}_{t_\star;m, l}$, we claim that $L$ is continuous on $[0,t_\star)$, where the deterministic time $t_\star>0$ is defined for $\widehat{L}$ as above. Towards a contradiction, suppose this is not the case. Then $\varsigma<t_\star$ for some of the sample paths on the event $(\theta W_t)_{t\leq t_\star} \in \mathcal{A}_{t_\star;m, l}$. Since $L$ is continuous on $[0,\varsigma)$ with $\varsigma>0$, letting $L$ take the place of $\widetilde{L}$ in the above comparison argument (for a given sample path $f(t)=\theta W_t$), we can deduce that $L \leq \widehat{L}$ on $[0,\varsigma)$ on the event $(\theta W_t)_{t\leq t_\star} \in \mathcal{A}_{t_\star;m, l}$. Since we have this uniform control on $L_t$ in terms of  $\widehat{L}_t$, for $t\in [0,\varsigma) \subseteq [0,t_\star]$, and since we also have uniform control over $W_t$, for $t\in [0,\varsigma)$, on the event $(\theta W_t)_{t\leq t_\star} \in \mathcal{A}_{t_\star;m, l}$, we can verify that the estimates in the proof of Lemma \ref{lem:density_control} apply uniformly for all $t<\varsigma$ (since $\varsigma<t_\star$), and thus we get $V_{\varsigma-}(x)<\alpha^{-1}$ for $x$ near the origin. But then Lemma \ref{lem:density_control} and jump minimality give that there cannot be jumps at time $\varsigma$ nor for some strictly positive (random) period of time thereafter which contradicts the definition of $\varsigma$ for any of the sample paths with $t_\star>\varsigma$. Consequently, we do indeed have $\varsigma\geq t_\star$ on the event $(\theta W_t)_{t\leq t_\star} \in \mathcal{A}_{t_\star;m, l}$.
	
	At this point, we have ensured that there is no blow-up on the deterministic interval $[0,t_\star]$, with $t_\star>0$, on the aforementioned event. From Proposition \ref{prop:no_blow_up_timing} we know that no blow-ups can occur strictly after time $\alpha^2/2\pi\sigma^2$. Thus, it remains to show that we can avoid blow-ups on $[t_\star,s_{\star})$ for $s_{\star}:=t_\star+\alpha^2/2\pi\sigma^2$. To this end, note that, for $t\in[t_\star,s_{\star})$,
	\begin{align*}
		\nu_{t-}( \hspace{0.4pt} [ 0,\alpha x] \hspace{0.4pt} ) 
		&\leq \mathbb{P}(X_{t-} \in [ 0, \alpha x]\mid \mathcal{F}_\infty) \nonumber \\
		&= \mathbb{P}(X_0 + \sigma B_t + \theta W_t - \alpha L_{t-} \in[ 0, \alpha x] \mid \mathcal{F}_\infty ) \nonumber \\
		&\leq \int^\infty_0 \int^{\alpha x}_0 \frac{1}{\sqrt{2\pi \sigma^2 t_\star }} \exp\Big\{ - \frac{(z - y - \theta W_t + \alpha L_{t-})^2}{2\sigma^2 s_{\star}} \Big\}  \mathrm{d}z \, d\bar{\nu}_0(y),
	\end{align*}
	since $B$, $\bar{X}_0$, and $W$ are independent, and $L$ is $W$-measurable. 
	Note that the above argument for no blow-up on $[0,t_\star]$ does not depend on $m>0$, so we are free to tune this parameter. In particular, we can take $m:=t_\star^{-1}K$ for a suitable constant $K>0$. 
	
	Now recall the family of paths $\mathcal{U}$ defined in \eqref{eq:mathcal_U}, and note that any sample path of $\theta W$ belonging to $\mathcal{A}_{t_\star;m, l} \cap \mathcal{U}_{ t_\star, s_{\star}; \alpha}$ satisfies $\theta W_t >  K - \alpha - \delta \varepsilon$ for $t\in[t_\star,s_\star]$. Therefore, taking $K>3\alpha+\delta \varepsilon$, we get
	\[
	\nu_{t-}( \hspace{0.4pt} [0, \alpha x] \hspace{0.4pt} ) \leq \frac{\alpha}{ \sqrt{2\pi \sigma^2 t_\star}} \exp\Big\{ -\frac{(K - 3\alpha - \delta\varepsilon)^2}{2\sigma^2 s_\star} \Big\} \cdot x,
	\]
	for $x\in(0,1)$, for every $t\in[t_\star,s_\star]$ on the event $(\theta W_s)_{s\in[0,s_\star]} \in  \mathcal{A}_{t_\star;m,l} \cap \mathcal{U}_{ t_\star, s_{\star}; \alpha}$.
	It follows that we can take $K$ large enough so that $\nu_{t-}(\hspace{0.4pt}[0, \alpha x]\hspace{0.4pt}) < x$, for $x\in(0,1)$, for all $t\in[t_\star, s_\star]$. This rules out blow-ups on $[t_\star, s_\star]$, by the minimality of jumps, and hence we can conclude that
	\[
	0 < \mathbb{P}\bigl( (\theta W_s)_{s\in [0,s_\star]} \in \mathcal{A}_{t_\star;m, l} \cap \mathcal{U}_{ t_\star, s_{\star}; \alpha}  \bigr) \leq \mathbb{P}(\varsigma = +\infty),
	\]
	as required, where we recall that $s_\star=t_\star+\alpha^2/2\pi\sigma^2$, $m=K/t_\star$, and $l=\delta\varepsilon$, for the small $t_\star>0$ found above and the fixed $\delta,\varepsilon>0$. Thus, the proof is complete.
\end{proof}


\bibliographystyle{alpha}

\end{document}